\documentclass[reqno]{amsart}
\usepackage{amssymb}
 
\usepackage[all,cmtip]{xy}

\numberwithin{equation}{section}

\theoremstyle{plain} 
\newtheorem{thm}[equation]{Theorem}
\newtheorem{cor}[equation]{Corollary}
\newtheorem{lem}[equation]{Lemma}
\newtheorem{claim}[equation]{Claim}
\newtheorem{prop}[equation]{Proposition}
\newtheorem{conj}[equation]{Conjecture}

\theoremstyle{definition}
\newtheorem{defn}[equation]{Definition}

\theoremstyle{remark}
\newtheorem{rem}[equation]{Remark}

\usepackage{amsmath,amsthm,amssymb}
\input xypic

\def\integers{{\mathbb{Z}}}
\def\reals{{\mathbb R}}
\def\rationals{{\mathbb Q}}

\def\hp2{{\mathbb{HP}^2}}

\def\disk{{\mathbb D}}
\def\sphere{{\mathbb S}}

\def\varep{\varepsilon}

\def\be{\begin{enumerate}}
\def\ee{\end{enumerate}}
\def\rank{\hbox{rank}}
\def\modd{\!\!\!\mod}
\def\bigno{\bigskip\noindent}

\def\calk{{\mathcal K}}

\def\limone{\varprojlim\,\!\!^1}

\def\medno{\medskip\noindent}

\def\ind{\mathrm{ind}\,}

\def\supp{\mathrm{Supp}}
\def\almostzero{
\left(
\begin{array}{cc} 
I&0\\
0&0
\end{array}
\right)
}
\def\almostzerotwo{
\left(
\begin{array}{cc} 
0&0\\
0&I
\end{array}
\right)
}

\def\matrixtwo{
\left(\begin{array}{cc} 1&0 \\ 0&0\end{array}\right)}

\begin{document}


\title[A new index theory for noncompact manifolds]
{Positive scalar curvature and a new index theory for noncompact manifolds}

\author{Stanley Chang, Shmuel Weinberger, Guoliang Yu}

\address{Department of Mathematics\\
  Wellesley College\\
  Wellesley, MA 02481}
  \email{schang@wellesley.edu}

\address{
   Department of Mathematics\\
   University of Chicago\\
   Chicago, IL 60637}
   \email{shmuel@math.uchicago.edu}

\address{
Department of Mathematics\\
Texas A \& M University\\
College Station, TX 77843\\
and the Shanghai Center for Mathematical Sciences, China}
\email{guoliangyu@math.tamu.edu}

\thanks{The first author is supported by a Wellesley College Brachman-Hoffman
Fellowship. The second and third authors are supported by NSF grants.}

\date{\today}


\begin{abstract}
In this article, we develop a new index theory for noncompact manifolds endowed with an
admissible exhaustion
by compact sets. This index theory allows us to provide examples of noncompact manifolds with exotic positive scalar curvature phenomena.
\end{abstract}

\maketitle

\tableofcontents

\section{Introduction}
\label{sec:intro}

\bigno
If $M$ is an $n$-dimensional manifold endowed with a Riemannian metric
$g$, then its scalar curvature $\kappa\colon M\to\reals$ satisfies
the property that, at each point $p\in M$, there is an expansion
\begin{equation*}
\mathrm{Vol}_M(B_\varepsilon(p))=\mathrm{Vol}_{\reals^n}(B_\varepsilon(0))\left(1
-\frac{\kappa(p)}{6(n+2)}\,\varepsilon^2+\cdots\right)
\end{equation*}
 for all
sufficiently small $\varepsilon>0$. A complete Riemannian metric $g$
on a manifold $M$ is said to have  uniformly positive scalar
curvature if there is a fixed constant $\kappa_0>0$ such that
$\kappa(p)\ge\kappa_0>0$ for all $p\in M$. For compact manifolds,
obstructions to such metrics are largely achieved in one of two
ways: (1) the minimal surface techniques in dimensions at most 7 by
Schoen-Yau \cite{Schoen-Yau-1979} and in dimension 8 by Joachim and
Schick \cite{Joachim-Schick-1998};
(2) the  Dirac index method for spin manifolds by Atiyah-Singer and
its generalizations by Connes-Moscovici, Hitchin, Gromov, Lawson, Roe 
and Rosenberg, among others.

\bigno In the realm of noncompact manifolds it is now well
recognized that the original approach by Gromov-Lawson
\cite{Gromov-Lawson-1980} and
Schoen-Yau \cite{Schoen-Yau-1979},
which proves that no compact manifold of nonpositive sectional curvature
can be endowed with a metric of positive scalar curvature, is
actually based on a restriction on the coarse quasi-isometry type of
complete noncompact manifolds. Connes and Moscovici 
\cite{Connes-Moscovici-1990} develop a higher
index theory that proves that any aspherical manifold whose fundamental group is
hyperbolic does not have a metric of positive scalar curvature. 
Roe \cite{Roe-1993} subsequently introduces
a coarse index theory to study positive scalar curvature problems
for noncompact manifolds. Block and Weinberger \cite{Block-Weinberger-1999}
investigate the problem of complete metrics for noncompact symmetric
spaces when no quasi-isometry conditions are imposed. They prove
that, if $G$ is a semisimple Lie group with maximal compact subgroup
$K$ and irreducible lattice $\Gamma$, then the double quotient
$M=\Gamma\backslash G/K$ can be endowed with a complete metric
of uniformly positive scalar curvature if and only if $\Gamma$ is an
arithmetic group with $\rank_\rationals\Gamma\ge3$. This theorem
includes, in light of the work of Borel and Harish-Chandra
\cite{Borel-Harish-Chandra-1962}, 
 previous results of
Gromov-Lawson \cite{Gromov-Lawson-1980}
 in rational rank 0 and 1. In the case when the
rational rank exceeds 2, Chang proves that any metric on $M$ with
uniformly positive scalar curvature fails to be coarsely equivalent to
the natural one \cite{Chang-2001}.

\bigno
The Gromov-Lawson-Rosenberg conjecture states that a spin closed manifold $M^n$
with $n\ge5$ has a metric of positive scalar curvature if, and only if,  its Dirac index vanishes
in $KO_\ast (C^\ast_r\pi)$, where $\pi=\pi_1(M)$. While this conjecture is known to
be false in general, it has been verified in number of cases. To study compact manifolds
$(M, \partial M)$
with boundary with respect to a positive scalar curvature metric that is collared at
the boundary, one would ideally like to produce a $C^\ast$-algebra that
encodes information about both $\pi_1(M)$ and $\pi_1(\partial M)$. In this
paper we show that such an algebra
can be constructed with the appropriate properties, and apply it to obtain information
about noncompact manifolds.

\bigno
In the first section, we use the notion of localization algebras
\cite{Yu-1997}
and generalized asymptotic morphisms to define a relative group $C^\ast$-algebra
$C^\ast_{max}(\pi_1(M), \pi_1(\partial M))$ along with a homomorphism
\begin{equation*}
\mu_{max}\colon KO_\ast(
M, \partial M)
\to KO_\ast(C^\ast_{max}(\pi_1(M), \pi_1(\partial M)))
\end{equation*} which we call the
maximal relative Baum-Connes map. The usual Baum-Connes conjecture has many different
guises, the simplest of which is that the
homomorphism 
$KO_\ast^\Gamma(\underline{E}\Gamma)\to KO_\ast(C^\ast_r\Gamma)$
is an isomorphism. One may similarly hope that the map $\mu_{max}$
above is an injection if $M$ and $\partial M$ are both aspherical.
In line with the compact case, we show that, if $M$ has a metric of positive
scalar curvature that is collared near the boundary, then
the relative index of the Dirac operator in $KO_\ast(M, \partial M)$ belongs to the kernel of
$\mu_{max}$. In this section, we also
formulate a relative Gromov-Lawson-Rosenberg conjecture for manifolds with boundary and
show that the converse to the above statement holds when the relative Gromov-Lawson-Rosenberg
conjecture holds for torsion-free amenable groups satisfying certain conditions
on their cohomological dimensions.

\bigno
In the next sections, we offer a new index theory for noncompact manifolds
with so-called {\it admissible exhaustions}.
We combine this theory with the machinery built in the first part
of the paper to give various geometric applications:
we first construct a noncompact manifold $M$ with an
exhaustion $\bigcup_{i=1}^\infty (M_i, \partial M_i)$ by compact sets with
boundary such that each $(M_i, \partial M_i)$ has a metric of positive scalar
curvature collared at the boundary, but $M$ itself has no metric of uniformly positive
scalar curvature. Next, we construct a noncompact manifold $N$ whose
space $PS(N)$ of uniformly positive scalar curvature metrics
has uncountably many connected components.

%
%

\bigno
A companion paper \cite{Chang-Weinberger-Yu-2014} will 
use the techniques of this paper and more complicated topology to
obtain a contractible manifold that has a positively curved exhaustion, but no
 metric of positive scalar curvature.

\bigno
The authors are grateful to the University of Chicago, the Shanghai Center
for Mathematical Sciences and the Mathematical Sciences Research Institute for hosting
their stays when the research for this paper was conducted. The authors thank John Roe and 
Jonathan Rosenberg for useful comments and conversations, as well as the referee for
his careful reading of this paper.
They also thank Eric Hilt
and Phil Hirschhorn
for help in document formatting.

\bigskip

\section{The relative group $C^\ast$-algebra and
the relative Gromov-Lawson-Rosenberg conjecture}

\bigno
In this section, we introduce the concept of relative group $C^\ast$-algebras and formulate a relative version of the Gromov-Lawson-Rosenberg conjecture. The $K$-theory of the relative group $C^\ast$-algebras serves as the receptacle of the relative higher index of the Dirac operators.

\medskip\noindent
 In this paper all $C^\ast$-algebras are real. We deal only with metric spaces $X$
that are locally compact and metrically locally simply connected;  i.e.~for all $\varep>0$
there is $\varep'\le\varep$ such that every ball in $X$ of radius $\varep'$ is simply
connected. 

\bigskip\noindent
If $G$ is a discrete group, denote by $C^\ast_r(G)$ and $C^\ast_{max}(G)$ the
usual reduced and maximal real $C^\ast$-algebras of $G$, respectively.
Let $Y\subseteq X$ both be compact (metric) spaces. We wish to define a Baum-Connes map
from the
relative $KO$-homology group $KO^{lf}_\ast(X,Y)$ to the $KO$-theory of some
relative $C^\ast$-algebra
encoding the fundamental groups of both $Y$ and $X$ and the homomorphism
between them. Let $\phi\colon C^\ast_{max}
(\pi_1(Y))\to C^\ast_{max}(\pi_1(X))$
be the map induced by the homomorphism $j_\ast
\colon\pi_1(Y)\to \pi_1(X)$. Consider the mapping
cone $C^\ast$-algebra of $\phi$
given by
\begin{equation*}
C_{\phi, max}=\{(a,f)\colon f\in C_0([0,1), C^\ast_{max}(\pi_1(X))), ~a\in C^\ast_{max}(\pi_1(Y)),
  f(0)=\phi(a)\}.
\end{equation*}
Define $C^\ast_{max}(\pi_1(X), \pi_1(Y))$ to be the 
seventh suspension $S^7C_{\phi, max}$ of $C_{\phi, max}$, i.e.~$C_{\phi, max}\otimes
C_0(\reals^7)$, where $C_0(\reals^7)$ is the $C^\ast$-algebra of continuous real-valued 
functions on $\reals^7$ which vanish at infinity.
The seventh suspension is chosen because $KO$-theory is eight-periodic.
We call this algebra the {\it maximal relative group $C^\ast$-algebra of $(\pi_1(X), \pi_1(Y))$}.
If in fact the homomorphism $j_\ast$ is an injection, we can likewise define a
{\it reduced relative $C^\ast$-algebra} $C^\ast_{red}(\pi_1(X), \pi_1(Y))$.

\bigno
If $M$ is a metric space, we say that a Hilbert space $H$ is an
{\it $M$-module} if there is a representation of the continuous
functions $C_0(M)$ in $H$, that is, a $C^\ast$-homomorphism
$C_0(M)\to B(H)$.  We will say that an operator $T : H\to H$ is
{\it locally compact} if, for all $\varphi\in C_0(M)$, the operators
$T\varphi$ and $\varphi T$ are compact on $H$. We define the {\it
support} $\mathrm{Supp}(\varphi)$ 
of $\varphi\in H$ as the smallest closed set $K\subseteq M$
such that, if $f\in C_0(M)$ and $f\varphi\ne0$, then $f\vert_K$
is not identically zero.
An operator $T\colon
H\to H$ on an $M$-module $H$ has {\it finite propagation}
if there is $R>0$ such that $\varphi T\psi =0$ whenever $\varphi$,
$\psi\in C_0(M)$ satisfy $d(\hbox{Supp}(\varphi ),\hbox{Supp}(\psi ))
>R$. 

\bigskip\noindent
Recall that a locally compact metric space $Z$ is said to have {\it bounded geometry}
if there is a discrete subset $Y\subseteq Z$ such that (1) $Y$ is $c$-dense for some $c\ge0$,
i.e.~$d(z,Y)\le c$ for all $z\in Z$; (2) for all $r>0$ there is $N$ such that, for all $p\in Y$,
we have $\#\{y\in Y\colon d(y,p)\le r\}\le N$.
In the remainder of the article, we assume that all spaces have bounded geometry.

\begin{defn} Let $Z$ be a locally compact metric space. If $H$ is a Hilbert space, we denote
by $B(H)$ the algebra of bounded operators with on $H$.
\begin{enumerate}
\item Denote by $\reals(Z)$ the Roe algebra, i.e.~the algebra of locally compact, finite
propagation operators on some ample $Z$-module $H$ by way of a representation
$\rho\colon C_0(Z)\to B(H)$ (see Roe \cite[Definition 4.5]{Roe-1993}).

\item Denote by $C^\ast_{red}(Z)$ and $C^\ast_{max}(Z)$ the completions of $\reals(Z)$
with respect to the reduced and maximal norm completions, respectively. Here
we define the maximal norm in the following way. If $a\in \reals(Z)$, then
let $||a||_{max}=\sup_\psi ||\psi(a)||$, where the supremum is taking over all
$\ast$-homomorphisms $\psi\colon \reals(Z)\to B(W)$, where $W$ is
real Hilbert space. By the bounded geometry assumption, the quantity $||a||_{max}$ is finite by 
Gong-Wang-Yu \cite[Lemma 3.4]{Gong-Wang-Yu-2008}. Note that,
if $Z$ is compact, then the two completions are the same and coincide with
$\mathcal{K}$, the $C^\ast$-algebra of compact operators, as $\reals(Z)$ is
already all of $\mathcal{K}$.
\item Let $\pi_1(Z)$ act on $\widetilde{Z}$ by deck transformations and
let $\reals(\widetilde{Z})^{\pi_1(Z)}$ be the algebra of operators in $\reals(
\widetilde{Z})$ that are invariant under this action. We endow $\reals(
\widetilde{Z})^{\pi_1(Z)}$ with a maximal norm by defining $||a||_{max}=
\sup_\psi ||\psi(a)||$, where the supremum is taken over all
$\ast$-homomorphisms $\psi\colon \reals(\widetilde{Z})^{\pi_1(Z)}\to B(H)$,
where $H$ is a Hilbert space.
Note that, although $\reals(\widetilde{Z})^{\pi_1(Z)}$ is a subalgebra
of $\reals(\widetilde{Z})$,  this maximal norm is different than the one defined in (2)
because the domain of $\psi$ is different. 
\end{enumerate}
\end{defn}

\begin{defn}
For continuous maps $g\colon [0,\infty)\to \reals(Z)$, we
 define norms $||g||_{red}=\sup_{t\in [0,\infty)}
||g(t)||_{red}$ and $||g||_{max}=\sup_{t\in[0,\infty)}||g(t)||_{max}$.
Suppose that (a) $g$ is uniformly bounded and uniformly continuous, and (b)
the propagation of $g(t)$ tends to 0 as $t\to\infty$. We define the following
sets:
\begin{enumerate}
\item Denote by $\reals_L(Z)$ the collection of maps $g$ satisfying (a) and (b).
\item Denote by  $C^\ast_{L, red}(Z)$ the closure
of $\reals_L(Z)$ with respect to $||\cdot||_{red}$, called the
{\it reduced localization algebra of $X$}.
\item
Denote by  $C^\ast_{L, max}(Z)$ the closure
of $\reals_L(Z)$ with respect to $||\cdot||_{max}$, called the
{\it maximal localization algebra of $X$}. Here the maximal norm is taken as
in (2) in the previous definition.
\item Denote by $C^\ast_{red}(\widetilde{Z})^{\pi_1(Z)}$ and $C^\ast_{max}(\widetilde{Z})^{\pi_1(Z)}$
the closure of the algebra $\reals(\widetilde{Z})^{\pi_1(Z)}$ with respect to the
reduced and maximal norms, respectively.  Here the maximal norm is taken as
in (3) in the previous definition.
\item Denote by $C^\ast_{L, red}(\widetilde{Z})^{\pi_1(Z)}$ and $C^\ast_{L, max}(\widetilde{Z})^{\pi_1(Z)}$
the closure of the algebra $\reals_L(\widetilde{Z})^{\pi_1(Z)}$ with respect to the
reduced and maximal norms, respectively.  Here the maximal norm is taken as
in (3) in the previous definition.
\end{enumerate}
\end{defn}

\begin{rem}
When $Z$ is compact, then the two localization algebras in (2) and (3) coincide.
\end{rem}

\medskip\noindent
For the rest of this paper, we will simplify notation and simply write $C^\ast_L(Z)$
for either the reduced and maximal localization algebra.

\medskip\noindent
Let $X$ be a locally compact metric space. We shall briefly recall the local index map $\ind_L
\colon KO_\ast^{lf}(X)\to KO_\ast(C_L^\ast(X))$,
first introduced by Yu in \cite{Yu-1997}.
We assume that $\ast\equiv 0 \modd 8$. The other cases can be handled in a
similar way with the help of suspensions. Here $KO_\ast^{lf}(X)\equiv KO^\ast(C_0(X))$.

\medskip\noindent
Let $(H, F)$ represent  a cycle for $KO_0^{lf}(X)$, where $H$ is a standard nondegenerate $X$-module and $F$ is a bounded operator acting on $H$ such that
$F^\ast F-I$ and $F F^\ast-I$ are locally compact, and $\phi F-F\phi$ is compact for all $\phi \in C_0(X).$
For each positive integer $n$, let $\{U_{n, i}\}_i$ be a locally finite and uniformly bounded open cover of
$X$ such that $\mathrm{diam}(U_{n,i})<\frac{1}{n}$.
Let $\{ \phi_{n,i}\}_i$ be a continuous partition of unity subordinate to the open cover.
Define
$$F(t)= \sum_i (( n-t) \phi_{n,i}^{ \frac{1}{2}} F \phi_{n,i}^{ \frac{1}{2}} + (n-t+1) \phi_{n+1,i}^{ \frac{1}{2}} F \phi_{n+1,i}^{ \frac{1}{2}} ) $$
for all positive integers $n$ and $t \in [n-1,n]$, where the infinite sum converges
in the strong topology.
If $\mathrm{prop}$ denotes the propagation of an operator (again see Roe
\cite[Definition 4.5]{Roe-1993}), then
notice that
$\mathrm{prop} (F(t))\rightarrow 0$ as $t\rightarrow \infty.$

\def\matrixtwo{
\left(\begin{array}{cc} 1&0 \\ 0&0\end{array}\right)}

\medskip\noindent
Observe that $F(t)$ is a multiplier of the localization algebra $C^\ast_L(X)$ and is invertible modulo the localization algebra.
Hence the standard index construction in $K$-theory gives
\begin{equation*}
\ind_L([(H, F)])=[P_{F}]
-\left[\matrixtwo\right]\in KO_0(C_L^\ast(X)),
\end{equation*}
 where $P_F$ is
an idempotent in the matrix algebra of $C_L^\ast(X)^+$.
We call this class $\ind_L([(H, F)])$
the {\it local index of $F$}.
In fact, we choose
$P_F(t)$ to be the matrix

{\footnotesize
\begin{equation*}
\left(
\begin{array}{cc}
F(t)F^\ast(t)+(1-F(t)F^\ast(t))F(t)F^\ast(t)) & F(t)(1-F^\ast(t)F(t))+(1-F(t)F^\ast(t))
 F(t)(1-F^\ast(t)F(t))\\
(1-F^\ast(t)F(t))F^\ast(t) & (1-F^\ast(t)F(t))^2\\
\end{array}
\right).
\end{equation*}
}

\noindent
See also Willett-Yu \cite{Willett-Yu-2012}. We write $P_F$ for $P_F(t)$ for simplicity.
 For the rest of this paper, we also abbreviate
$[(H, F)]$ as $[F]$ and
$\ind_L [(H, F)]$ as $\ind_L [F]$.

\medskip\noindent
The following isomorphism is demonstrated  in Yu \cite[Theorem 3.2]{Yu-1997} in the case when $ X$  is a CW complex and for general metric space $X$
in Qiao-Roe \cite[Theorem 3.4]{Qiao-Roe-2010}.

\begin{prop}
\label{yu-localization} 
The local index map  $\ind_L\colon KO_\ast(X)\to KO_\ast(C_L^\ast(X))$ is an isomorphism.
\end{prop}

\begin{defn}
Let $Y\subseteq X$ be compact metric spaces.
In the definitions of $C_L^\ast(Y)$ and $C_L^\ast(X)$, we choose the 
$Y$-module and $X$-module to be 
$\ell^2(Z_Y)\otimes H$ and $\ell^2(Z_X)\otimes H$ such that $Z_Y\subseteq Z_X$ are
countable dense subsets of $Y$ and $X$, respectively, and $H$ is a separable and infinite-dimensional
Hilbert space. The inclusion isometry from $\ell^2(Z_Y)\otimes H$ to $\ell^2(Z_X)\otimes H$
induces a homomorphism $i\colon C_L^\ast(Y)\to C_L^\ast(X)$.
\end{defn}

\begin{rem}
The homomorphism $i$ is not canonical. However it induces a homomorphism at the $KO$-theory
that is canonical.
\end{rem}

\noindent
Let $C_i$
be the mapping cone of $i$ given by
\begin{equation*}
C_i=\{(a,f)\colon f\in C_0([0,1), C^\ast_L(X)),  \,a\in C^\ast_L(Y),\,
  f(0)=i(a)\}.
\end{equation*}
Define the {\it relative $KO$-homology group of $(X,Y)$} to be
$KO_\ast(X, Y)\equiv KO_\ast(S^7C_i)$.

\noindent
This definition of relative $KO$-homology allows us to have a relative long exact
sequence 
\begin{equation*}
\cdots\to KO_\ast(Y)\to KO_\ast(X)\to KO_\ast(X,Y)\to\cdots.
\end{equation*}

\begin{lem}
\label{lem-roe}
Let $X$ be a compact space and let $\mathcal{K}$ be the $C^\ast$-algebra
of compact operators on a separable, infinite-dimensional Hilbert space.
Then there is an isomorphism
$C^\ast_{red}(\widetilde{X})^{\pi_1(X)}\cong C^\ast_r(\pi_1(X))\otimes\mathcal{K}$
and
$C^\ast_{max}(\widetilde{X})^{\pi_1(X)}\cong C^\ast_{max}(\pi_1(X))\otimes\mathcal{K}$.
\end{lem}

\begin{proof}
In Roe \cite[Lemma 2.3]{Roe-2002} the
 $\ast$-isomorphism 
$(\reals\widetilde{X})^{\pi_1(X)}\cong (\reals\pi_1(X))\otimes \calk$ is proved.
This algebraic $\ast$-isomorphism extends to the required
$\ast$-isomorphism in both the reduced and maximal case, since $\calk$ is a
nuclear $C^\ast$-algebra.
\end{proof}

\begin{prop}
\label{prop-lifting}
Let $X$ be a compact metric space with universal cover $\widetilde{X}$. There is
$\varep>0$ depending only on $X$ such that, if $b$ is an operator
in $\reals(X)$ with propagation at most $\varep$,
 then $b$ lifts to a $\pi_1(X)$-invariant operator
$\widetilde{b}$ in $\reals(\widetilde{X})$.
\end{prop}

\begin{proof}
In the definition of $\reals(X)$, we choose the $X$-module to be
$\ell^2(Z_X)\otimes H$ such that $Z_X$ is a countable dense subset of $X$
and $H$ is a separable and infinite-dimensional Hilbert space.
Let $p\colon \widetilde{X}\to X$ be the projection map. We define
$Z_{\widetilde{X}} = p^{-1}(Z_X)$. We choose the $\widetilde{X}$-module
to be $\ell^2(Z_{\widetilde{X}})\otimes H$ in the definition of $\reals(\widetilde{X})$.
Every operator $b\in \reals(X)$ can be represented by a kernel $k(\cdot, \cdot)$ such that
$k(x,y)$ belongs to $\mathcal{K}$ for all $(x,y)\in Z_X\times Z_X$ and $\mathrm{Supp}(k)$
is contained in $\{(x,y)\in X\times X\colon d(x,y)<r\}$ for some $r>0$.  The smallest
such $r$ is the propagation of $b$. Now let $k'(x', y')=k(p(x'), p(y'))$ for all 
$(x', y')\in Z_{\widetilde X}\times Z_{\widetilde X}$ satisfying $d(x', y')<r$
and $k'(x',y')=0$ for all $(x',y')\in Z_{\widetilde X}\times Z_{\widetilde X}$
satisfying $d(x',y')\ge r$. By the compactness of $X$, there
is $\varep>0$ such that, if $b$ has propagation at most $\varep$, then
 $k'$ represents an element $\widetilde{b}$
of $\reals(\widetilde{X})$ and $\widetilde{b}$ has the same propagation as $b$.

\bigno
This discussion shows that there exists $\varep >0$ such that,
  if $b\in \reals(X)$ and $\mathrm{prop}(b)<\varep$, then there is
a unique lifting of $b$ in  $\reals(X)$ to $\phi(b)$ in $ \reals (\widetilde{X})$. 
\end{proof}

\noindent
Note that, 
if the propagations $\mathrm{prop}(b_1), \mathrm{prop}(b_2)<\varep/2$, then
this lifting respects
multiplication and addition, i.e.~$\phi(b_1b_2)=\phi(b_1)\phi(b_2)$ and
$\phi(b_1+b_2)=\phi(b_1)+\phi(b_2)$.

\begin{defn}
Let $s\in [0,\infty)$ and let $X$ be a compact metric space.
 For all $b\in \reals_L(X)$, denote by $b_s\in \reals_L(X)$
the operator given by $b_s(t)=b(s+t)$ for all $t\in [0,\infty)$.
Let $\varep$ be as in the above proposition.
For each $b\in \reals_L(X)$, there is $s_b>0$ such that $\mathrm{prop}(b_s)<\varep$
when $s>s_b$.
 We define $\phi_s(b)=\widetilde{b}_s\in \reals_L(\widetilde{X})^{\pi_1(X)}$ 
when $s>s_b$.
\end{defn}

\noindent
The next result indicates that $\phi_s$ is an asymptotic morphism in the
following generalized sense.

\begin{lem}
\label{generalized-asymptotic-morphism}
Let $X$ be a compact metric space.
For all $b\in\reals_L(X)$, let $s_b$ be given as in the previous definition.
\begin{enumerate}
\item 
There is $C>0$ such that, for all $b\in \reals_L(X)$, if $s>s_b$, 
then 
\begin{equation*}
||\phi_s(b)||_{red}\le C||b||_{red}~~\mathrm{and}~~||\phi_s(b)||_{max}
\le C||b||_{max}.
\end{equation*} 
\item For all $b\in \reals_L(X)$, if $s>s_b$, then $\phi_s(b)^\ast=\phi_s(b^\ast)$.
\item For all $b_1, b_2\in\reals_L(X)$, the operator
\begin{equation*}
\phi_s(b_1b_2)-\phi_s(b_1)\phi_s(b_2)
\end{equation*}
is zero when 
$s>\max\{s_{b_1}, s_{b_2}, s_{b_1b_2}\}$.
\end{enumerate}
\end{lem}

\begin{proof}
Let $\{U_i\}_{i=1}^N$ be a finite open cover of $X$ such that, for each $i$, the diameter of the union
of all $U_j$ satisfying $U_j\cap U_i\ne\emptyset$ is less than $\varep$, where
$\varep$ is as in Proposition \ref{prop-lifting}. Let $\{\varphi_i\}_i$ be the continuous
partition of unity subordinate to $\{U_i\}$. We have $\phi_s(b)=\sum_{i=1}^N \phi_s(\varphi_ib)$.
By the definition of $\phi_s$ and the choice of $\phi_i$, we have $||\phi_s(\varphi_ib)||=||\varphi_ib||
\le ||b||$. It follows that $||\phi_s(b)||\le N||b||$ if $s>s_b$. This proves (1). The proofs of (2) and (3) are
straightforward.
\end{proof}

\bigskip\noindent
There is a pushdown
 $\reals_L(\widetilde{X})^{\pi_1(X)}
\to\reals_L(X)$ for operators with small propagation. Such a pushdown induces homomorphisms
\begin{equation*}
 KO_\ast(C_{L, max}^\ast(\widetilde{X})^{\pi_1(X)})\to 
KO_\ast(C^\ast_L(X))
\end{equation*}
 and
 \begin{equation*}
 KO_\ast(C_{L, red}^\ast(\widetilde{X})^{\pi_1(X)})\to 
KO_\ast(C^\ast_L(X)),
\end{equation*}
which are inverses to the homomorphisms induced by the liftings.
This lemma implies that the liftings $\phi_s$ induce isomorphisms
$KO_\ast(C^\ast_L(X))\to KO_\ast(C_{L, max}^\ast(\widetilde{X})^{\pi_1(X)})$ and
$KO_\ast(C^\ast_L(X))\to KO_\ast(C_{L, red}^\ast(\widetilde{X})^{\pi_1(X)})$.

\begin{defn}
Let $j_\ast\colon \pi_1(Y)\to \pi_1(X)$ be the homomorphism induced by the inclusion $Y\to X$.
Then $j_\ast$ induces a unique map $\eta\colon \widetilde{Y}\to \widetilde{X}$ such that 
$\eta(gy)=i_\ast(g)\eta(y)$ for all $g\in \pi_1(Y)$ and $y\in \widetilde{Y}$. 
\end{defn}

\noindent
Note that such $\eta$ exists because $X$ and $Y$ are metrically locally simply connected.

\bigskip\noindent
Let $p$ be the covering map $\widetilde{X}\to X$ and let $Y'=p^{-1}(Y)$. 
 Let $p'\colon \widetilde{Y}\to Y$
be the covering map from the universal cover $\widetilde{Y}$.
Let $Y''$ be the Galois coveirng
of $Y$ whose deck transformation group is $j_\ast\pi_1(Y)$, and let $p''\colon Y''\to Y$ be the
covering map. 

\bigskip\noindent
We have $Y'=\pi_1(X)\times_{j_\ast\pi_1(Y)}Y''$. This decomposition gives rise to a natural
$\ast$-homomorphism $$\psi'\colon C^\ast_{max}(Y'')^{j_\ast\pi_1(Y)}\to  C_{max}^\ast(Y')^{\pi_1(X)}.$$
Choose countable dense subsets $Z_Y$ of $Y$ and $Z_X$ of $X$ such that $Z_Y\subseteq Z_X$.
Let $H$ be a separable and infinite-dimensional
Hilbert space.
We use the modules $\ell^2(p^{-1}(Z_Y))\otimes H$, $\ell^2(p^{-1}(Z_X))\otimes H$, 
$\ell^2((p')^{-1}(Z_Y))\otimes H$ and $\ell^2((p'')^{-1}(Z_Y)\otimes H$, respectively, to define
$C_{max}^\ast(Y')^{\pi_1(X)}$, $C_{max}^\ast(\widetilde{X})^{\pi_1(X)}$, $C_{max}^\ast(\widetilde{Y})^{\pi_1(Y)}$
and $C_{max}^\ast(Y'')^{j_\ast\pi_1(Y)}$.

\begin{lem}
\label{big-lemma}
There exists a $\ast$-homomorphism $$\psi''\colon C_{max}^\ast(\widetilde{Y})^{\pi_1(Y)}\to C_{max}^\ast(Y'')^{j_\ast\pi_1(Y)}$$
such that there is $\varep>0$ for which, if $k\in C^\ast(\widetilde{Y})^{\pi_1(Y)}$ is an operator
with propagation at most $\varep$ and is represented as a kernel $k$ on $(p')^{-1}(Z_Y)$ with values in 
$\calk$, then there is a unique kernel  $k_Y$  on $Z_Y$  with values in 
$\calk$ such that
$k(x,y)=k_Y(p(x),p(y))$ for all $x,y\in p^{-1}(Z_Y)$ satisfying  $d(x,y)\le \varep$
and  $\psi''(k)$ is represented by a kernel $k''$ on $(p'')^{-1}(Z_Y)$ with
values in $\calk$ such that $k''(x,y)
=k_Y(p''(x), p''(y))$ for all $x,y\in (p'')^{-1}(Z_Y)$ satisfying $d(x,y)\le \varep$.
\end{lem}

\begin{proof}
Let $H$ be the kernel of the homomorphism $j_\ast\colon \pi_1(Y)\to \pi_1(X)$. Let $k$
be an operator in $\reals(\widetilde{Y})^{\pi_1(Y)}$ represented by a kernel $k(x,y)$
on $(p')^{-1}(Z_Y)$. We define a kernel $k_a(x,y)$ on $(p')^{-1}(Z_Y)$
 by the formula $k_a(x,y)=\sum_{h\in H}k(hx,y)$ for all $x,y\in (p')^{-1}(Z_Y)$.
Note that the above sum is finite since $k$ has finite propagation. We have $k_a(h_1x, h_2y)=k_a(x,y)$
for all $h_1, h_2\in H$ and $x,y\in (p')^{-1}(Z_Y)$. For each $x,y\in (p')^{-1}(Z_Y)$, let
$[x], [y]$ be the corresponding pair of equivalence classes in $(p'')^{-1}(Z_Y)=(p')^{-1}(Z_Y)/H$. We let
$k''([x], [y])=k_a(x,y)$. Note that $k''$ is well-defined. We now define a $\ast$-homomorphism
$\psi''\colon \reals(\widetilde{Y})^{\pi_1(Y)}\to \reals(Y'')^{j_\ast\pi_1(Y)}$ given by $\psi''(k)=k''$. 
By maximality, this map $\psi''$ extends to a $\ast$-homomorphism $C_{max}^\ast(\widetilde{Y})^{\pi_1(Y)}
\to C_{max}^\ast(Y'')^{j_\ast\pi_1(Y)}$. We choose $\varep>0$ small enough such that $d(hx,x)>10\varep$ for all
$h\ne e$ in $H$ and all $x\in\widetilde{Y}$. If $d([x], [y])>\varep$, then $d(hx,y)>\varep$ for all $h\in H$. 
Therefore if $k$ has propagation at most $\varep$, then $k''$ has propagation at most $\varep$. If $\varep$
is small enough, there is a unique kernel $k_Y$ on $Z_Y$ such that $k(x,y)=k_Y(p(x), p(y))$ for all $x,y\in 
p^{-1}(Z_Y)$ satisfying $d(x,y)\le\varep$ and $k_Y$ has propagation at most $\varep$. If $d(x,y)\le \varep$,
then $d(hx,y)<\varep$ for all $h\ne e$ in $H$ and $x,y\in\widetilde{Y}$. Therefore $k_a(x,y)=k(x,y)$
if $d(x,y)\le\varep$. It follows that $k''(x,y)=k_Y(p''(x), p''(y))$ for all $x,y\in (p'')^{-1}(Z_Y)$ satisfying
$d(x,y)\le \varep$. 
\end{proof}

\bigskip\noindent
Let $\psi''$ be as in Lemma \ref{big-lemma} above and let $\psi'$ be as previously defined.
We now define a $\ast$-homomorphism 

\begin{equation}
\label{psi-max-formula}
\psi_{max}=\psi'\circ\psi''\colon C_{max}^\ast(\widetilde{Y})^{\pi_1(Y)}
\to C_{max}^\ast(\widetilde{X})^{\pi_1(X)}.
\end{equation}
This homomorphism  in turn induces a $\ast$-homomorphism
$$\psi_{L, max}: \quad C^\ast_{L, max}(\widetilde{Y})^{\pi_1(Y)} \rightarrow C^\ast_{L, max}(\widetilde{X})^{\pi_1(X)}.$$
Let $C_{\psi_{L, max}}$ be the mapping cone of $\psi_{L, max}$ given
by $$C_{\psi_{L, max}}
=\{(a,f)\colon f\in C_0([0,1), C^\ast_{L, max}(\widetilde{X})^{\pi_1(X)}), ~a\in
C^\ast_{L, max}(\widetilde{Y})^{\pi_1(Y)}, 	
  f(0)=\psi_{L, max}(a)\}.$$
Recall that $i\colon C_L^\ast(Y)\to C_L^\ast(X)$ is the homomorphism induced
by the inclusion $Y\to X$,
and $C_i$ is its mapping cone.
For each $(b,f)\in C_i$ satisfying $\mathrm{prop}(b)<\infty$ and $\mathrm{prop}{f(t)}<\infty$
for all $t\in[0,1]$, there is $s_{(b,f)}>0$ such that $\mathrm{prop}(b_s)<\varep$ and
$\mathrm{prop}(f(t))<\varep$ for all $s>s_{(b,f)}$.
We define 
\begin{equation*}
\chi_{s, max}(b,f)=(\phi_s(b_s), \phi_s(f(\cdot)_s))\in C_{\psi_{L, max}}
\end{equation*}
 for all $s>s_{(b,f)}$, where $\phi_s$ is as in  Lemma \ref{generalized-asymptotic-morphism}. 
By the same lemma, we then know that
$\chi_{s, max}$ induces a homomorphism 
\begin{equation*}
(\chi_{s, max})_\ast\colon 
KO_\ast(S^7C_i)\to KO_\ast(S^7C_{\psi_{L, max}}).
\end{equation*}

\medskip\noindent
Let  $e$ be the evaluation homomorphism induced by the evaluation maps at $0$ from
$C^\ast_{L, max}(\widetilde{X})^{\pi_1(X)}$ to $C^\ast_{ max}(\widetilde{X})^{\pi_1(X)}$
and from $C^\ast_{L, max}(\widetilde{Y})^{\pi_1(Y)}$ to  $C^\ast_{ max}(\widetilde{Y})^{\pi_1(Y)}$.
This homomorphism induces a map $e_\ast\colon KO_\ast(S^7C_{\psi_{L,max}})\to
KO_\ast(S^7C_{\psi_{ max}})$ at the level of $KO$-theory.

\medno
Define $\mu_{max}$ to be the composition given by
\begin{equation*}
KO_\ast(S^7C_i)
\xrightarrow{(\chi_{s, max})_\ast} KO_\ast(S^7C_{\psi_{L, max}})\xrightarrow{e_\ast}
KO_\ast(S^7C_{\psi_{max}}).
\end{equation*}
Equivalently $\mu_{max}$ is a map
\begin{equation*}
\mu_{max}\colon
 KO_\ast(X,Y)
\to
KO_\ast(C^\ast_{max}(\pi_1(X), \pi_1(Y)))
\end{equation*}
 which we
call the {\it maximal relative Baum-Connes map.}
%
A {\it reduced relative Baum-Connes map}
\begin{equation*}
\mu_{red}\colon KO_\ast(X,Y)\to  KO_\ast(
C^\ast_{red}(\pi_1(X), \pi_1(Y)))
\end{equation*}
can be similarly constructed if the homomorphism $j$ from $\pi_1(Y)$ to $\pi_1(X)$ is injective.

\begin{rem}
\label{remark-lifting}
 By homotopy invariance, both $\mu_{max}$
and $\mu_{red}$ are independent of the choices of the liftings.
\end{rem}

\begin{conj}
Let $Y\subseteq X$ and suppose that $X$ and $Y$ are both aspherical compact spaces.
\begin{enumerate}
\item (Relative Novikov conjecture)
The maximal relative Baum-Connes map
\begin{equation*}
\mu_{max}\colon KO_\ast(X,Y)\to  KO_\ast(
C^\ast_{max}(\pi_1(X), \pi_1(Y)))
\end{equation*} is an injection.
\item (Relative Baum-Connes conjecture)
If $j\colon \pi_1(Y)\to \pi_1(X)$ is an injection,
then the reduced relative Baum-Connes map
\begin{equation*}
\mu_{red}\colon KO_\ast(X,Y)\to  KO_\ast(
C^\ast_{red}(\pi_1(X), \pi_1(Y)))
\end{equation*}
 is an isomorphism.
\end{enumerate}
\end{conj}

\begin{rem}
If the classic Baum-Connes conjecture holds for $\pi_1(X)$ and $\pi_1(Y)$, then
statement (2) is true for the pair $(\pi_1(X), \pi_1(Y))$. In general the maximal
relative Baum-Connes conjecture may not be an isomorphism.
The real version ($KO$)
of the Baum-Connes
conjecture follows from the classic (complex version) of the Baum-Connes conjecture
(see Baum-Karoubi \cite{Baum-Karoubi-2004}). After inverting 2, even the injectivity
of the complex Baum-Connes map implies the injectivity of the real Baum-Connes map (see
Schick \cite[Corollary 2.13]{Schick-2004}).
\end{rem}

\noindent
Recall that the notion of $K$-amenability was formulated by Cuntz \cite[Definition 2.2]{Cuntz-1983}.
This notion can be extended to the $KO$-setting.

\begin{thm}
\label{amenable}
Suppose that $Y\subseteq X$ are aspherical compact spaces such that $\pi_1(Y)$ and
$\pi_1(X)$ are $K$-amenable and satisfy the Baum-Connes conjecture.
\begin{enumerate}
\item Then $\mu_{max}$ is an isomorphism.
\item Assume also that $\pi_1(Y)\to \pi_1(X)$
is an injection.  Then $\mu_{red}$ is an isomorphism.
\end{enumerate}
\end{thm}

\begin{proof}
By the definition of $K$-amenability,
the natural homomorphisms \linebreak
$C^*_{max}(\pi_1 (X)) \rightarrow C^*_{r}(\pi_1 (X)) $ and $C^*_{max}(\pi_1 (Y))
\rightarrow C^*_{r}(\pi_1 (Y))$
induce $KK$-equivalences.
If $\pi_1(X)$ and $\pi_1(Y)$ are $K$-amenable and satisfy the
Baum-Connes conjecture, and if $\pi_1(Y)$ injects
into $\pi_1(X)$,
then the $KO$-theory of the reduced relative group $C^*$-algebra coincides
with the $KO$-theory of the maximal
relative group $C^*$-algebra.

\medskip\noindent
The theorem is proved from the following commutative diagram and
the five-lemma.

\begin{equation*}
{\footnotesize
\xymatrix{
KO_{n+1}(Y)\ar[r]\ar[d] & KO_{n+1}(C^\ast(\pi_1(Y)))\ar[d]\\
KO_{n+1}(X)\ar[r]\ar[d] & KO_{n+1}(C^\ast(\pi_1(X)))\ar[d]\\
KO_{n+1}(X, Y) \ar[r]\ar[d] & KO_{n+1}(C^\ast(\pi_1(X), \pi_1(Y)))\ar[d]\\
KO_n(Y) \ar[r]\ar[d] & KO_{n}(C^\ast(\pi_1(Y)))\ar[d]\\
KO_{n}(X)\ar[r] & KO_{n}(C^\ast(\pi_1(X)))\\
}
}
\end{equation*}
\end{proof}

\bigno
We now prove that the existence of positive scalar curvature implies that a particular
index vanishes in the $KO$-theory of the relative group $C^\ast$-algebra. For the rest of this section,
the $C^\ast$-algebras involved are  maximal. If the reduced relative group $C^\ast$-algebra is well defined, then the rest of this section extends to the reduced case as well. We  will use $C^*(\pi_1(X), \pi_1(Y))$ to
denote both the reduced and maximal relative group $C^*$-algebra when the use of such a notation does
not cause confusion.

\bigno
Let $M$ be a spin manifold with boundary $N=\partial M$.
We assume that the dimension of $M$ is $0\modd 8$. The other cases can be handled in a similar way
with the help of suspensions.
More specifically, in dimension $k$ mod 8 for some $0\leq k<8$, we consider the manifold $M\times \reals^{8-k}$. We can define a relative higher index of the Dirac operator associated to the space  $M\times \reals^{8-k}$  in 
$KO_0(C^\ast(\pi_1(M), \pi_1(\partial M))\otimes C^\ast_L(\reals^k))$. 
We can apply the same argument below to show that this  relative index vanishes  
$KO_0(C^\ast(\pi_1(M), \pi_1(\partial M))\otimes C^\ast_L(\reals^k))$
if $(M, \partial M)$ is a compact spin manifold with boundary endowed with
a metric of positive scalar curvature that is collared at the boundary.
This relative higher index corresponds to the relative index of the Dirac operator associated to $M$ under the isomorphism
$$KO_0(C^\ast(\pi_1(M), \pi_1(\partial M))\otimes C^\ast_L(\reals^k))\cong KO_k(C^\ast(\pi_1(M), \pi_1(\partial M))) .$$
As a consequence, the  relative index of the Dirac operator associated to $M$  vansihes in 
$ KO_k(C^\ast(\pi_1(M), \pi_1(\partial M)))$ if $(M, \partial M)$ is a compact spin manifold with boundary endowed with
a metric of positive scalar curvature that is collared at the boundary.

\bigno
Extend the manifold by attaching a cylinder $W=N\times [0,\infty)$ to the boundary,
forming a noncompact manifold $Z$.
Let $D$ be the Dirac operator on $Z$.
Let $f$ be an odd smooth chopping function in the sense of Roe on the real line satisfying the following conditions:
(1) $|f(x)|\le1$ for all $x$ and
$f(x)\rightarrow \pm 1$ as $x \rightarrow \pm \infty$; (2) $g=f^2-1 \in S(\mathbb{R})$, the space of Schwartz functions, (3) if $\widehat{f}$ and $\widehat{g}$ are the Fourier transforms
of $f$ and $g$, respectively, then
$\supp (\widehat{f})\subseteq [-1, 1]$ and $\supp (\widehat{g} ) \subseteq [-2,2]$. Such a chopping
function exists (cf.~Roe \cite[Lemma 7.5]{Roe-1988}).
We define $$F_D= f(D)= \frac{1}{2\pi} \int_{-\infty}^{\infty} \widehat{f}(t) \exp(itD) dt .$$
By condition (3) above, it follows that the propagation of $F_D$ is at most 1.
 Let $[F_D]$ be its homology class in $KO_0^{lf}(Z)=KO^0(C_0(Z))$. We simplify the notation by replacing
$F_D$ with $F$ and $P_{F_D}$ with $P_D$. We write
$$\ind_L([F])=[P_{D}]
-\left[\almostzero\right]\in KO_0(C_L^\ast(Z)),$$ where $P_D$ is
an idempotent in the matrix algebra of $C_L^\ast(Z)^+$
and $\ind_L$ is the local index map.
The element $P_{D}-\almostzero$ belongs to the matrix algebra of the localization algebra
$C_L^\ast(Z)$.  

\bigno
Let $v$ be an invertible element in the matrix algebra of $C_0(\reals^7)^+$ representing
the generator in $KO_{-1}(C_0(\reals^7))\cong KO_0(C_0(\reals^8))$ (see Atiyah \cite{Atiyah-1966}
or Schr\"oder \cite[Proposition 1.4.11]{Schroeder-1993}).
 Let $\tau_D=v\otimes P_D+I\otimes(I-P_D)$. Then we have
$\tau_D^{-1}=v^{-1}\otimes P_D+I\otimes (I-P_D)$.  If $\chi_M$ is
the characteristic function on $M$, let $\tau_{D,M}=
(1\otimes\chi_M)\tau_D(1\otimes\chi_M)$ and $(\tau_D^{-1})_M=
(1\otimes\chi_M)\tau_D^{-1}(1\otimes\chi_M)$. 
In the future pages, we will simply write $\chi_M$ for $1\otimes \chi_M$.
For all $s\in [0,1]$, define $w_{D,M}(s)$ to be the product
\begin{equation*}
\left(
\begin{array}{cc}
I& (1-s)\tau_{D,M}\\0&I
\end{array}
\right)
\left(
\begin{array}{cc}
I& 0\\-(1-s)(\tau_D^{-1})_M&I
\end{array}
\right)
\left(
\begin{array}{cc}
I& (1-s)\tau_{D,M}\\0&I
\end{array}
\right)
\left(
\begin{array}{rr}
0& -I\\I&0
\end{array}
\right).
\end{equation*}
Define
\begin{equation*}
q_{D,M}(s)=w_{D,M}(s)\almostzero
w^{-1}_{D,M}(s).
\end{equation*}
Now define $C^\ast_L(N\subseteq M)$
be the closed two-sided ideal of $C^\ast_L(M)$ generated by $C^\ast_L(N)$
considered as a subalgebra of $C^\ast_L(M)$.
Then $\tau_{D,M}$ and $(\tau^{-1}_D)_M$ both lie in $C^\ast_L(M)\otimes C_0(\reals^7)$.  Both
$\tau_{D,M}(\tau^{-1}_D)_M-I$ and $(\tau^{-1}_D)_M\tau_{D,M}-I$ lie in $C^\ast_L(N\subseteq M)$.
As a consequence $q_{D,M}(0)$ is an element in the matrix algebra
of $(C_L^\ast(N\subseteq M)\otimes C_0(\reals^7))^+$. We define $[q_D]$ to be the $KO$-theory element 
\begin{equation*}
[(q_{D, M}(0), q_{D, M}(\cdot))]-
\left[\almostzero\right]
\end{equation*}
in $KO_0(S^7C_j)$, where $C_j$ is the mapping cone associated with 
$j\colon C_L^\ast(M, N)\to C_L^\ast(M)$ and $S^7C_j=C_j \otimes  C_0(\reals^7)$. The inclusion map 
$C^\ast_L(N)\to C_L^\ast(N\subseteq M)$ induces an isomorphism
\begin{equation*}
KO_\ast(C_L^\ast(N))\cong KO_\ast(C^\ast_L(N\subseteq M)),
\end{equation*} 
which can be proved by a standard Mayer-Vietoris sequence and a five-lemma argument. 
As a consequence,
we have the isomorphism 
$KO_0(S^7C_j)\cong KO_0(M,N)$.

\bigno
We call the class $[q_D]$ the {\it relative $KO$-homology class of $D$.} We
define the
{\it relative higher index of $D$} to be $\mu(q_D)\in KO_0(C^\ast(\pi_1(M), \pi_1(N)))$.

\vskip -2cm
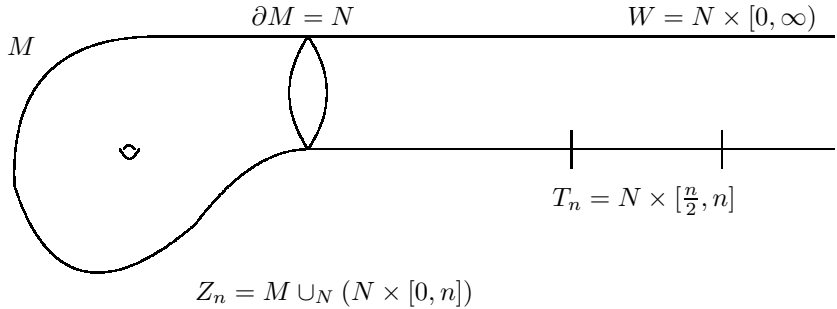
\begin{center}
\setlength{\unitlength}{1cm}
\begin{picture}(15,7)
\put(4,3.5){\line(1,0){7}}
\put(4,5){\line(1,0){7}}
\put(7.5,3.25){\line(0,1){.5}}
\put(9.5,3.25){\line(0,1){.5}}
\qbezier(4,5)(4.5,4.25)(4,3.5)
\qbezier(4,5)(3.5,4.25)(4,3.5)
\put(4,5){\line(-1,0){2}}
\qbezier(2,5)(0,5)(0.1,3)
\qbezier(0.1,3)(0.75,1)(2.5,2.5)
\qbezier(2.5,2.5)(3.25,3.5)(4,3.5)
\qbezier(1.5,3.5)(1.625,3.25)(1.75,3.5)
\qbezier(1.55,3.5)(1.6,3.6)(1.7,3.5)
\put(8.25,5.15){$W=N\times[0,\infty)$}
\put(3.25,5.15){$\partial M=N$}
\put(0,4.75){$M$}
\put(2.5,1.5){$Z_{n}=M\cup_{N}(N\times[0,n])$}
\put(7.25,2.75){$T_{n}=N\times[\frac{n}{2},n]$}
\end{picture}
\end{center}

\vskip -.5cm
\begin{thm}
\label{index-theorem}
If $(M, \partial M)$ is a compact spin manifold with boundary endowed with
a metric of positive scalar curvature that is collared at the boundary, then the relative higher index of the Dirac operator
is zero in $KO_*(C^\ast(\pi_1(M),
\pi_1(\partial M)))$.
\end{thm}

\begin{proof}
As before, let $N=\partial M$ and $Z=M\cup_N (N\times [0,\infty))$.
 Denote by $Z_n$ and $Z_n'$  the truncations
$Z_n=M\cup_N (N\times[0,n])$,  $Z_n'=M\cup_N (N\times[0,\frac{n}{2}])$, and let $T_n$ be the subset of $Z_n$
given by $T_n=N\times\left[\frac{n}{2}, n\right]$.
We assume that the dimension of $Z$ is $0 \modd 8
$. The other cases can be handled in a similar way
with the help of suspensions (refer back to the section after
Theorem \ref{amenable}).

\medskip\noindent
Let $u\in [1,\infty)$ and write $$\ind_L(uD)=[P_{uD}]-\left[\almostzero\right]
\in KO_0(C_L^\ast(Z)).$$

\bigno
We define $w_{D, Z_n}(s)$ and $q_{D, Z_n}(s)$
by replacing $M$ with $Z_n$ in the definitions of $w_{D,M}(s)$ and $q_{D, M}(s)$, respectively,
before Theorem \ref{index-theorem}.
By the finite propagation of $D$, we know that the propagation of
$\exp(i tD)$ is less than or equal to $|t|$. It follows that the propagation of
$P_{uD}$ is less than or equal to $100u$. This estimate is based on the matrix formula
before Proposition \ref{yu-localization}.

\begin{claim}
\label{long-claim1}
For all $u>0$, there exists $N_u>0$ such that, for all $n>N_u$, we have
$$\chi_{Z_n'} \left( 
q_{uD, Z_n}(0)-\almostzero \right)\chi_{Z_n'} =0,$$
$$ \chi_{T_n} \left( 
q_{uD, Z_n}(0)-\almostzero \right)\chi_{Z_n'} =0,$$
$$ \chi_{Z_n'} \left( 
q_{uD, Z_n}(0)-\almostzero \right)\chi_{T_n}=0.$$
\end{claim}

\begin{proof}
Let $\alpha=\tau_{uD, Z_n}$ and $\beta=(\tau^{-1}_{uD})_{Z_n}$. 
We can compute 
\begin{equation*}
w_{uD, Z_n}(0)\almostzero w^{-1}_{uD, Z_n}(0)
\end{equation*}
\begin{equation*}
=
\left(
\begin{array}{cc}
(2\alpha-\alpha\beta\alpha)\beta & (2\alpha-\alpha\beta\alpha)(I-\beta\alpha)\\
(I-\beta\alpha)\beta & (I-\beta\alpha)^2
\end{array}
\right).
\end{equation*}
We note that $(2\alpha-\alpha\beta\alpha)\beta = \alpha(I-\beta\alpha)\beta
+(\alpha\beta-I)+I$. Let $N_u=100u$.
Using the formulas for $\alpha$ and $\beta$, and the
fact that $P_{uD}$ has propagation at
most $100u$, we know that the elements 
$\chi_{Z_n'}(\alpha\beta-I)$,  $\chi_{Z_n'}(\beta\alpha-I)$, $(\alpha\beta-I)\chi_{Z_n'}$,
$(\beta\alpha-I)\chi_{Z_n'}$, $(I-\beta\alpha)\beta\chi_{Z_n'}$ and 
$\chi_{Z_n'}\alpha(I-\alpha\beta)$ are all zero when $n>N_u$. 
Now our claim follows.
\end{proof}

\noindent
Let $P_{uD}^{(n)}=\chi_{Z_n}P_{uD}\chi_{Z_n}$, where $\chi_{Z_n}$ is
the characteristic function on $Z_n$.
By the construction of $P_{uD}$ we have $||P_{uD}||\le 10$.  As a result, we have
$||P^{(n)}_{uD}||\le 10$, giving an upper bound for $||q_{uD, Z_n}||$.
Together with the above claim, this implies that
$\left[\prod_n
q_{uD, Z_n}(0)\right] \in
\prod_n (S^7C_L^\ast (Z_n))^+/\bigoplus_n (S^7C_L^\ast (Z_n))^+$
belongs to the image of the inclusion map
$$\prod_n (S^7C_L^\ast (T_n))^+/\bigoplus_n (S^7C_L^\ast (T_n))^+ 
   \rightarrow \prod_n (S^7C_L^\ast (Z_n))^+/\bigoplus_n (S^7C_L^\ast (Z_n))^+, $$
where $(S^7C_L^\ast (Z_n))^+$ and $(S^7C_L^\ast (T_n))^+$ are respectively obtained from  $S^7C_L^\ast (Z_n)$ and $S^7C_L^\ast (T_n)$ by adjoining a unit.
We identify $\left[\prod_n q_{uD, Z_n}(0)\right]$ with an element in \linebreak
$\prod_n (S^7C_L^\ast (T_n))^+/\bigoplus_n (S^7C_L^\ast (T_n))^+.$
Now
$\left(\prod_n q_{uD, Z_n}(0), \prod_n q_{uD, Z_n}(s)\right)$
gives an element in the matrix algebra of $ \prod_n (S^7C_{j_n})^+/\bigoplus_n (S^7C_{j_n})^+$,
where $s\in [0,1]$ is the variable and $C_{j_n}$ is the mapping cone
of the homomorphism
 $ j_n\colon S^7C_L^\ast (T_n)  \rightarrow S^7C_L^\ast (Z_n)$.

\medskip\noindent
Recall that $W=N\times [0,\infty)$ and let
 $W'=N\times\reals$ be the double of $W$.
Let $D'$ be the Dirac operator on $W'$. Let $\widetilde{W'}$ be the universal cover of $W'$
and $\widetilde{D'}$ be the lifting of $D'$ to $\widetilde{W'}$.
Let $\widetilde{D}$ be the lifting of $D$ to $\widetilde{Z}$,
the universal cover of  $Z$.  We know that $P_{u \widetilde{D'}}(0)$ and $ P_{u\widetilde{D}}(0)$
are respectively the liftings of
$P_{u D'}(0)$ and $P_{uD}(0)$. Define
$$x_{n, u}(s) = q_{u\widetilde{D}, \widetilde{Z}_n}(s),$$
where $P^{(n)}_
{u\widetilde{D}} = \chi_{\widetilde{Z}_n}P_{u\widetilde{D}}\chi_{\widetilde{Z}_n}$.
Define $$ y_{n,u} =  q_{u\widetilde{D}', \widetilde{W}'_n}(0),$$
where $W'_n = N\times (-\infty, n]$.
By an argument similar to the proof of Claim \ref{long-claim1}, we know that
$\left[\prod_n y_{n, u} \right]$ is an operator in the image of the inclusion map:
$$\prod_n (S^7C^\ast (\widetilde{T}_n)^{\pi_1( N)})^+/\bigoplus_n
 (S^7 C^\ast (\widetilde{T}_n)^{\pi_1( N)})^+ \rightarrow
\prod_n (S^7C^\ast (\widetilde{W'})^{\pi_1( N)})^+/\bigoplus_n ( S^7C^\ast
(\widetilde{W'})^{\pi_1( N)})^+.$$ We identify $\left[\prod_n y_{n, u} \right]$
with an element in $\prod_n (S^7C^\ast (\widetilde{T}_n)^{\pi_1( N)})^+/\bigoplus_n (S^7 C^\ast (\widetilde{T}_n)^{\pi_1( N)})^+$.

\bigskip\noindent
By Lemma \ref{big-lemma} and Formula \ref{psi-max-formula}, there
is a natural $\ast$-homomorphism $$\phi_n\colon
C^\ast(\widetilde{T}_n)^{\pi_1(N)}\to C^\ast(T_n')^{\pi_1(M)}.$$ 
For each $n$, the map $\phi_n$ induces a natural $\ast$-homomorphism
$$ S^7C^\ast(\widetilde{T}_n)^{\pi_1(N)}\to S^7C^\ast(\widetilde{Z}_n)^{\pi_1(M)},$$ which
we still denote by $\phi_n$. 
We have  $$\left[\prod_n \phi_n(y_{n, u}) \right] = \left[\prod_n x_{n, u}(0) \right]$$
 in
$\prod_n (S^7C^\ast (\widetilde{Z}_n)^{\pi_1( M)})^+/\bigoplus_n  
(S^7C^\ast (\widetilde{Z}_n)^{\pi_1( M)})^+.$
  Denote by $C_{\phi_n}$ the mapping cone of the map $\phi_n$.
The element $\prod_n (y_{n, u}, x_{n, u}(s))$ gives a $KO$-theory element
$$\left[\prod_n (y_{n, u}, x_{n, u}(s))\right]$$ in
$KO_0 \left( \prod_n S^7C_{\phi_n}/\bigoplus_n S^7C_{\phi_n}\right).$

\bigskip\noindent
Let $V_{1,n}\colon L^2[0,n]\to L^2[0,1]$ be the isometry given by $f(\cdot)\mapsto
\frac{1}{\sqrt{n}}f(n\cdot)$ for all $f\in L^2[0,n]$. Let $V_{2,n}\colon L^2[\frac{n}{2},n]
\to L^2[\frac{1}{2},1]$ be the
isometry given by $f(\cdot)\mapsto \frac{1}{\sqrt{n}}f(n\cdot)$ for all $f\in L^2[0,\frac{n}{2}]$.
We can naturally construct isometries $V_{1,n}'\colon L^2(\widetilde{Z}_n)\to L^2(\widetilde{Z}_1)$
and $V_{q,n}'\colon L^2(\widetilde{T}_n)\to L^2(\widetilde{T}_1)$. Conjugations by $V_{1,n}'$
and $V_{2,n}'$ give us a $\ast$-isomorphism $C_{\phi_m}\to C_{\phi_1}$,
where $C_{\phi_1}$ is naturally $\ast$-isomorphic to $C_\phi$.
Identifying $(S^7C_{\phi_m})^+$ with $(S^7C_{\phi})^+$ 
for all $n$, we see that  $[(y_{n,u}, x_{n,u}(s))]=[\mu(q_D)]$ in 
$KO_0(S^7C_\phi)$. It follows that there is a
natural isomorphism
$$\psi\colon KO_0  \left(   \prod_n S^7C_{\phi_n}/\bigoplus_n S^7C_{\phi_n}  \right)\to KO_0
\left( \prod_n S^7C_{\phi}/\bigoplus_n S^7C_{\phi} \right) $$ such that
\begin{equation*}
\psi \left(\left[\prod_n (y_{n, u}, x_{n, u}(s))\right]\right)=\left[\prod_n \mu(q_D)\right],
\end{equation*}
where $\mu(q_D) \in KO_0 (S^7C_{\phi})
\cong KO_0(C^\ast(\pi_1(M), \pi_1(N)))$ was defined as the relative higher index of $D$
before Theorem \ref{index-theorem}
and $KO_0( \prod_n S^7C_{\phi}/\bigoplus_n S^7C_{\phi})$
is identified with
\begin{equation*}
\prod_n KO_0(S^7C_{\phi})/\bigoplus_n KO_0 (S^7C_{\phi}).
\end{equation*}
When $M$ has a metric of uniform positive scalar curvature, then by the 
Lichnerowicz formula we know that $P^{(n)}_{u\widetilde{D}}(0)$ converges to 
$\almostzero$ in the operator norm as $u\to \infty$,  $n\to\infty$ and $n\ge N_u$. 
As a consequence, we know that
\begin{equation*}
\tau_{u\widetilde{D}, \widetilde{W}_n'}\to v\otimes\almostzero + I\otimes 
\almostzerotwo
\end{equation*}
and 
\begin{equation*}
 y_{n,u} \to 
\exp\left(2\pi i
\almostzero\right)=I
\end{equation*}
 in the operator norm as $u\to\infty$, $n\to\infty$ and $n\ge N_u$. 
Together with the formula for $\chi_{n,u}(s)$, we then have
 $\left[\prod_n(y_{n,u}, x_{n,u}(s))\right]=0$
in $KO_0\left( \prod_n S^7C_{\phi}/\bigoplus_n S^7C_{\phi} \right)$. 
Therefore it follows that $\mu(q_D)=0$.
\end{proof}

\bigno
As mentioned in the introduction,
the Gromov-Lawson-Rosenberg conjecture states that a closed spin manifold $M^n$
with $n\ge5$ has a metric of positive scalar curvature if, and only if, its Dirac index vanishes
in $KO_\ast (C^\ast_r\pi)$, where $\pi=\pi_1(M)$. We formulate now a relative version
of this conjecture.

\begin{conj}
\label{conjecture-relative-glr} 
(Relative Gromov-Lawson-Rosenberg)
Let $(N, \partial N )$ be a compact spin manifold with boundary. Let $n$  be the dimension of $N$.
 If the relative higher index of the Dirac operator $D$ is zero in
$KO_n( C^\ast (\pi_1(N), \pi_1 (\partial N)))$,
  then there is a metric of positive scalar curvature on
$N$ that is collared near $\partial N$.
\end{conj}

\begin{rem}
\leavevmode
\begin{enumerate}
\item This conjecture is made by analogy to surgery theory, where obstructions to
surgery for degree one normal maps have this formal structure.
\item By the Gromov-Lawson
surgery theorem
\cite[Theorem A]{Gromov-Lawson-1980} and Schoen-Yau \cite[Corollary 6]{Schoen-Yau-1979}
 as improved
by Gajer \cite{Gajer-1987}, 
the $\pi$-$\pi$ case of the conjecture is correct.
\item  Because of the the failure
of stability for the ordinary Gromov-Lawson conjecture
Schick \cite[Example 2.2]{Schick-1998} and Dwyer-Schick-Stolz \cite{Dwyer-Schick-Stolz-2003}),
we recognize that, in general,  this statement
cannot  be true as stated.  One should either
interpret it as a stable conjecture (i.e.~after crossing
with some number of Bott manifolds, see Rosenberg-Stolz
\cite{Rosenberg-Stolz-1995}) or as a guide to prove the
correct statement in the unstable situation.
We hope to address this matter in a future paper.
\end{enumerate}
\end{rem}

\bigno
In Rosenberg-Stolz \cite{Rosenberg-Stolz-1995}, the index map $\alpha\colon
\Omega^{spin}_n(B\pi)\to KO_n(C_r^\ast\pi)$ is factored in the following way:
\begin{equation*}
\Omega_n^{spin}(B\pi)\xrightarrow{D_\ast} ko_n(B\pi)\xrightarrow{p} KO_n(B\pi)
\xrightarrow{A} KO_n(C^\ast_r\pi)
\end{equation*}
where $p\colon ko_n(B\pi)\to KO_n(B\pi)$ is the canonical map from connective
to periodic $KO$-homology and $A$ is the standard assembly map. This sequence
can be generalized to pairs. Let $(N, \partial N)$ be a manifold with boundary and
let $\pi=\pi_1(N)$ and $\pi^\infty=\pi_1(\partial N)$. Then we have a composition
\begin{equation*}
\Omega_n^{spin}(B\pi, B\pi^\infty)\xrightarrow{D_\ast} ko_n(B\pi, B\pi^\infty)
\xrightarrow{p} KO_n(B\pi, B\pi^\infty)
\xrightarrow{A} KO_n(C^\ast(\pi, \pi^\infty)).
\end{equation*}
Let $\mathrm{Pos}_n^{spin}(B\pi, B\pi^\infty)$ be the subgroup of
$\Omega_n^{spin}(B\pi, B\pi^\infty)$ consisting of bordism classes
represented by pairs $(M^n, \partial M^n, f)$ for which $M$ admits a metric of positive
scalar curvature that is collared near the boundary.

\bigno
There is a map from $\partial M$ to $B\pi^\infty$ classifying its universal
cover $\widetilde{\partial M}$.  
By elementary homotopy
theory, the composite map to $B\pi$ commutes up to homotopy with the map
$M \to B\pi$ classifying its universal cover $\widetilde{M}$. 
The homotopy extension principle then implies that  we have a
map of pairs $(M, \partial M) \to (B\pi, B\pi^\infty)$.

\begin{thm}
\label{thm-pos}
Let $(M, \partial M)$ be a spin manifold with boundary of dimension $\ge6$.
Let $\pi=\pi_1(M)$ and $\pi^\infty=\pi_1(\partial M)$. Let $u\colon (M, \partial M)
\to (B\pi, B\pi^\infty)$ be the map described above. Then
$(M, \partial M)$ has a positive scalar curvature which is collared near the boundary
$\partial M$ if, and only if,  the index
$D_\ast[(M, \partial M), u]$ lies in $\mathrm{Pos}^{ko}_n(B\pi, B\pi^\infty)$,
where $\mathrm{Pos}^{ko}_n(B\pi, B\pi^\infty)$ is the image of $D_\ast$ restricted
to $\mathrm{Pos}_n^{spin}(B\pi, B\pi^\infty)$.
\end{thm}

\begin{proof}
First we will explain that the capacity of a spin manifold $M^n$  for $n\ge6$ to admit a positive  
scalar curvature metric depends only on its spin cobordism class.  As in the closed case, this result
follows from the Gromov-Lawson surgery
 theorem, or equivalently the reduction to spin cobordism.  
For manifolds $(M, \partial M)$ with boundary whose 
boundary is collared, there is a relative
surgery theorem that follows from an improvement by Gajer \cite{Gajer-1987}
of the usual
 Gromov-Lawson surgery theorem that provides a collared neighborhood
 for the trace of the surgery. We remind the reader that
the Gromov-Lawson theorem
holds if the spin cobordism respects fundamental group and the dimension is at least
5. The proof of our theorem requires two applications of the
Gajer/Gromov-Lawson Theorem, as we will now demonstrate. 
Let $(M, \partial M)$ and $(M', \partial M')$ be cobordant and suppose
that $(M, \partial M)$ has a metric of positive scalar curvature that is collared
near the boundary. 

\medno
We first apply the surgery theorem to the restricted cobordism between $\partial M$
and $\partial M'$ 
to endow $\partial M'$ with a metric of positive scalar curvature. 
 Conner \cite{Conner-1979} allows us to execute a rounding maneuver that
replaces metric pieces isometric to a product $\partial M \times Q$
of the boundary with 
 a quadrant $Q$ with metric pieces isometric to the product $\partial M \times H$ 
of the boundary with a half-space $H$ (and vice versa). 
Finally we use the Gromov-Lawson theorem to endow $M'$ with
a metric of positive scalar curvature that is collared near $\partial M'$.
We note that the original proof of the Gromov-Lawson theorem is  sufficiently local as to be 
unchanged by this slight additional generality.

\medno
As a next step, we need to show 
that all the elements of the kernel of the map
from  relative spin bordism to relative $ko$ have positive scalar curvature, i.e.~that 
$\mathrm{ker}\, D_\ast\subseteq\mathrm{Pos}^{spin}_n(B\pi,
B\pi^\infty)$. Both away from the prime 2 and at the prime 2, the inclusion can be
obtained from the relative versions of existing theorems. Away from 2, the result
holds by readapting the result of F\"uhring \cite{Fuehring-2012} on Baas-Sullivan
theory. This result was stated in Rosenberg-Stolz \cite{Rosenberg-Stolz-2001}
as unpublished  work of Jung. F\"uhring
proves that a smooth spin closed manifold $M$ of dimension $n\ge5$ admits a metric
of positive scalar curvature if its orientation class in $ko_n(B\pi)$ lies in the
subgroup consisting of elements which contain positive representatives. At the prime
2, we can extend Theorem B (2) of Stolz \cite{Stolz-1994}. Here he proves the following.
Let $X$ be a topological space.
 Suppose that $T_n(X)$
is the subgroup of $\Omega_n^{spin}(X)$ consisting of bordism classes $[E, f\circ p]$,
where $p\colon E\to B$ is an $\hp2$-bundle over a spin closed manifold $B$ of dimension
$n-8$ and $f$ is a map $B\to X$. Then the map $\Omega_n^{spin}(X)/T_n(X)\to ko_n(X)$
is a 2-local isomorphism. In the papers of
both F\"uhring and Stolz it is effectively shown that the kernel of $D_\ast$
is a homology theory. As such we can extend these results to pairs.
\end{proof}

\begin{cor}
Let $p\colon ko_n(B\pi, B\pi^\infty)\to KO_n(B\pi, B\pi^\infty)$
and $$A\colon KO_n(B\pi, B\pi^\infty)\to KO_n(C^\ast(\pi_1, \pi_1^\infty))$$
be as above, with $n\ge6$.
The Relative Gromov-Lawson-Rosenberg  conjecture holds if $p$ and $A$
are both injective.
\end{cor}


\begin{thm}
Let $n\ge6$. Let $N^n$ be a manifold with boundary such that
$\pi_1(N)$ and $\pi_1(\partial N)$ are both amenable. Suppose that $\pi_(\partial N)\to 
\pi_1(N)$ is an injection and that
the cohomological dimensions of $\pi_1(N)$ and $\pi_1(\partial N)$ are
less than $n$.  If the
classifying spaces $B\pi_1(N)$ and $B\pi_1(\partial N)$
are finite complexes, then the
 Relative Gromov-Lawson-Rosenberg conjecture  holds for the
pair $(N, \partial N)$.
\end{thm}


\begin{proof}
Let $A=\pi_1(N)$ and $B=\pi_1(\partial N)$. 
The $E^2$ term for the Atiyah-Hirzebruch spectral sequence for
$KO_n(K(A,1), K(B, 1))$  is $H_p(A, B; KO_q)$. Similarly the $E^2$
term for 
$ko_n(K(A,1), K(B, 1))$ is $H_p(A, B; ko_q)$. The groups coincide when $q\ge0$.
There is a comparison map between the spectral sequences
from the $ko$-sequence to the $KO$-sequence which is an isomorphism
on $E^2$ for $q\ge0$. The reason that this map may fail to be an isomorphism on
$E^\infty$ is that there are differentials for the $KO$-sequence
that can start in the fourth quadrant and end in the first.
For this reason, a nonzero element in $ko_n$ can vanish in $KO_n$.
 But if $n > \max\{cd(A), cd(B)\}$, differentials can only come
from the line $p+q = n+1$ with $p\le \max\{cd(A), cd(B)\}$. But then $q$ is positive and
the map is therefore an isomorphism.


\medskip\noindent
Using Higson-Kasparov \cite[Theorem 1.1]{Higson-Kasparov-2001} extended into the
$KO$ setting,
we see that the $KO$-theory groups of $C^\ast_{max}(\pi)$ and
$C^\ast_{max}(\pi^\infty)$ are given by
the $KO$-theories of their classifying spaces. Thus the relative
assembly map $A\colon KO_n(B\pi, B\pi^\infty)\to KO_n(C^\ast
(\pi_1, \pi_1^\infty))$
 is an isomorphism. The rest of the proof is as the last paragraph
 of Theorem \ref{thm-pos}.
 \end{proof}

\begin{rem}
 This unstable version of Conjecture \ref{conjecture-relative-glr} for large $n$
obviously implies the stable version of the conjecture for all $n$.
\end{rem}
%

\bigskip
\section{A new index theory for noncompact manifolds}
\label{sec:exhaustion}

\bigno
In this section we will develop a new index theory for a noncompact manifold.
Our index theory will depend on a choice of an exhaustion.

\begin{defn}
 Let $(Y, d)$ be a noncompact, complete metric space. Suppose that
$Y$ is also metrically locally simply connected; i.e.~for all $\varep>0$ there is
$\varep'\le\varep$ such that every ball in $X$ of radius $\varep'$ is simply connected.
Let   $Y_1\subseteq Y_2\subseteq Y_3\subseteq\cdots$
 be a sequence of connected compact subsets of $Y$. We say that $\{Y_i\}$ is an  {\it admissible
exhaustion} if
\begin{enumerate}
\item $Y=\bigcup_{i=1}^\infty Y_i$;
\item for each $j>i$, there is a connected compact subset $Y_{i,j}\subseteq Y$
such that $Y_j= Y_{i,j}\cup Y_i$ and $Y_{i,j}\cap Y_i=\partial Y_i$, where $\partial Y_i
=Y_i-\mathring{Y}_i$ for all $i$ and $\mathring{Y}_i$ denotes the interior of $Y_i$;
\item $d(\partial Y_i, \partial Y_j)\to \infty$ as $|j-i|\to\infty$. 
\end{enumerate}
Often we will write $\{Y_i; Y_{i,j}\}$ for the exhaustion.
\end{defn}

\noindent
Let $\{Y_i; Y_{i,j}\}$ be an admissible exhaustion of $Y$.
Define $D^\ast_i$ to be the $C^\ast$-algebra inductive limit
given by
\begin{equation*}
D^\ast_i\equiv\lim_{j\to\infty,\, j>i} C^\ast_{max}(\pi_1(Y_j), \pi_1(Y_{i,j}))\otimes \calk,
\end{equation*}
where $\calk$ is the $C^\ast$-algebra of compact operators.
Let $$ \prod_{i=1}^\infty D_i^\ast = \left\{(a_1, a_2, \ldots) \colon a_i \in D_i^\ast,   ~\sup_i||a_i||<\infty\right\} .$$
There
is a natural homomorphism $\rho_{i+1}\colon D^\ast_{i+1}\to D_i^\ast$
induced by the group homomorphisms given by inclusions of the corresponding
spaces. Let $\rho$ be the homomorphism from $ \prod_{i=1}^\infty D_i^\ast$
 to $ \prod_{i=1}^\infty D_i^\ast$
mapping $(a_1, a_2, \ldots)$ to $(\rho_2(a_2), \rho_3 (a_3), \ldots)$.

\bigno
We now define the $C^\ast$-algebra $A(Y)$ by:
\begin{equation*}
A(Y)\equiv \left\{a \in C\left([0, 1], \prod_{i=1}^\infty D_i^\ast\right)\colon
\rho(a(0))=a(1) \right\}.
\end{equation*}
Notice that $A(Y)$ is the $C^\ast$-algebra inverse
limit of the $D_i^\ast$ in a certain homotopical sense. In particular, this
$C^\ast$-algebra encodes dynamical information about how the
fundamental groups of the pieces of the exhaustion interact with each other. We emphasize
that the definition of $A(Y)$ depends on the exhaustion $\{Y_i\}$ of $Y$.
 We will now define an index map
$\sigma\colon  KO^{lf}_\ast(Y)=KO^0(C_0(Y))\to KO_\ast(A(Y))$.

\bigno
There exists $\varep_0>0$  such that, for any closed subspace $Z$
of $Y$,  any operator on a $Z$-module with propagation less than or equal
to $\varep_0>0$ can be lifted to the universal cover of $Z$.
One can prove that the above constant $\varep_0$ exists because $Y$ is metrically
locally simply connected (as defined in the beginning of Section 2). The proof is similar
to that of Proposition \ref{prop-lifting}.

\medskip\noindent
If an operator $F$ represents  a class in $KO_0^{lf}(Y)$, for each $\varep <\frac{\varep_0}{100}$,
 we can choose another operator $F_\varep$ representing the same $KO$-class such that
the propagation of $F_\varep$ is smaller than   $\varep.$
Let $$\ind_L ([F_\varep]) = [P_{F_\varep}]- \left[\almostzero\right]
\in KO_0 (C_L^\ast(Y)),$$ where $P_{F_\varep}$ is the idempotent in the matrix
algebra of $C_L^\ast(Y)^+$
as given in the definition of the local index
whose propagation is less than $100\varep<\varep_0$.

\noindent
Let $P_{F_\varep}^{(j)}=\chi_{Y_j} P_{F_{\varep}} \chi_{Y_j}$ and
let $\widetilde{P}_{F_\varep}^{(j)}$ be the lifting of $P_{F_\varep}^{(j)}$ to $\widetilde{Y}_{j}$, the universal cover of $Y_j$.
Let $v$ be an invertible element in the matrix algebra of $C_0(\reals^7)^+$ representing
the generator in $KO_{-1}(C_0(\reals^7))\cong KO_0(C_0(\reals^8))$
(see Atiyah \cite{Atiyah-1966} or Schr\"oder \cite[Proposition 1.4.11]{Schroeder-1993}). Let
$\tau_{F_\varep}^{(j)}=v\otimes \widetilde{P}^{(j)}_{F_\varep} +I\otimes(I-
\widetilde{P}^{(j)}_{F_\varep})$
and let $(\tau^{-1}_{F_\varep})^{(j)} = 
v^{-1}\otimes \widetilde{P}^{(j)}_{F_\varep} +I\otimes(I-
\widetilde{P}^{(j)}_{F_\varep})$.
For all $s\in [0,1]$, define $w^{(j)}_{F_\varep}(s)$ to be the product
\begin{equation*}
\left(
\begin{array}{cc}
I& (1-s)\tau^{(j)}_{F_\varep}\\0&I
\end{array}
\right)
\left(
\begin{array}{cc}
I& 0\\-(1-s)(\tau^{-1}_{F_\varep})^{(j)}&I
\end{array}
\right)
\left(
\begin{array}{cc}
I& (1-s)\tau^{(j)}_{F_\varep}\\0&I
\end{array}
\right)
\left(
\begin{array}{rr}
0& -I\\I&0
\end{array}
\right).
\end{equation*}

\bigno
For each $k$, there exist $j_k>k$ and a sequence 
of positive numbers $\{\varep_k\}$ converging to $0$ such that $100\varep_k<\varep_0$  and $y_k=w^{(j_k)}_{F_{\varep_k}}(0)\almostzero (w^{(j_k)}_{F_{\varep_k}}(0))^{-1}$ has propagation less than $\varep_0$ for
all $k$, and there is $z_k$ in the matrix algebra of 
$(S^7C_{max}^\ast(  \widetilde{Y}_{k, j_k} ) ^{\pi_1 (Y_{k, j_k})})^+$
such that $y_k=\phi_{k, j_k}(z_k)$,  where $\phi_{k, j_k}  $ is the $\ast$-homomorphism
from the matrix algebra
of $(S^7C_{max}^\ast(  \widetilde{Y}_{k, j_k} ) ^{\pi_1 (Y_{k, j_k})})^+$ to  
the matrix algebra of $(S^7C_{max}^\ast(
\widetilde{Y}_{j_k} )^{\pi_1 (Y_{j_k})})^+$. 
Note that the existence of such a $\ast$-homomorphism follows from Lemma \ref{big-lemma} and
Formula \ref{psi-max-formula}.
The existence of such $z_k$ is a result of the
following claim.


\begin{claim}
\label{long-claim2}
Let $\widetilde{Y}_{j_k}$ be the universal cover of 
$Y_{j_k}$ and let $\pi_k\colon \widetilde{Y}_{j_k}\to Y_{j_k}$ be the covering map. Then
we have
 \begin{equation*}
y_k
=\chi_{\pi^{-1}_k(Y_{k, j_k})}
y_k\chi_{\pi^{-1}_k(Y_{k, j_k})}\oplus \chi_{\widetilde{Y}_{j_k}-\pi^{-1}_k(Y_{k, j_k})}
\end{equation*}
 when $k$ and $j_k$ are sufficiently large.
\end{claim}

\begin{proof}
This proof is identical to that of Claim \ref{long-claim1}.
\end{proof}

\medskip\noindent
Let $\lambda\in[0,1]$.
We define $z'_k(\lambda)$ by replacing $P_{F_{\varep_k}}^{(j_k)}(0)$ with
 $(1-\lambda) P_{F_{\varep_k}}^{(j_k)}(0)+\lambda P_{F_{\varep_{k+1}}}^{(j_{k+1})}(0)$
in the definition of  $z_k$. Define $y_k'(\lambda)=\phi_{k,j_k}(z_k'(\lambda))$.
Let $\psi_k$ be the natural homomorphism
$\psi_k\colon S^7C_{max}^\ast(\widetilde{Y}_{j_k})^{\pi_1(Y_{j_k})}
\to S^7C_{max}^\ast(\widetilde{Y}_{j_{k+1}})^{\pi_1(Y_{j_{k+1}})}$.
Let 
\begin{equation*}
\tau_k(\lambda)=v\otimes \left((1-\lambda)\psi_k(\widetilde{P}^{(j_k)}_{F_{\varep_k}}(0))
+\lambda \widetilde{P}^{(j_{k+1})}_{F_{\varep_{k+1}}}(0)\right) +I\otimes (I-
 \left((1-\lambda)\psi_k(\widetilde{P}^{(j_k)}_{F_{\varep_k}}(0))+\lambda
\widetilde{P}^{(j_{k+1})}_{F_{\varep_{k+1}}}(0))\right)
\end{equation*}
and 
\begin{equation*}
\tau'_k(\lambda)=v^{-1}\otimes \left((1-\lambda)\psi_k(\widetilde{P}^{(j_k)}_{F_{\varep_k}}(0))
+\lambda \widetilde{P}^{(j_{k+1})}_{F_{\varep_{k+1}}}(0)\right) +I\otimes (I-
 \left((1-\lambda)\psi_k(\widetilde{P}^{(j_k)}_{F_{\varep_k}}(0))+\lambda
\widetilde{P}^{(j_{k+1})}_{F_{\varep_{k+1}}}(0))\right)
\end{equation*}
for all $\lambda \in [0,1]$. Again, the existence of $\psi_k$ follows from Lemma \ref{big-lemma}
and Formula \ref{psi-max-formula}.
For all $s, \lambda\in [0,1]$, define $(w_k(s))(\lambda)$ to be the product
\begin{equation*}
\left(
\begin{array}{cc}
I& (1-s)\tau_k(\lambda)\\0&I
\end{array}
\right)
\left(
\begin{array}{cc}
I& 0\\-(1-s)\tau'_k(\lambda)&I
\end{array}
\right)
\left(
\begin{array}{cc}
I& (1-s)\tau_k(\lambda)\\0&I
\end{array}
\right)
\left(
\begin{array}{rr}
0& -I\\I&0
\end{array}
\right).
\end{equation*}

\medskip\noindent
Define
\begin{equation*}
(c_k(s))(\lambda) =(w_k(s))(\lambda)\almostzero ((w_k(s))(\lambda))^{-1}.
\end{equation*}

\noindent 
By the definition of $z_k'$ and $c_k$, 
for each $\lambda\in[0,1]$,
the pair $(z'_k(\lambda), (c_k ( \cdot) )(\lambda))$ lies in $(S^7D_k^\ast)^+$, 
where $ D_k^\ast$ is as in the definition of $A(Y)$.
Let $a_k= (z_k', c_k)$.  We finally define the index of $F$ in $KO_0(A(Y))$ to be
$$\sigma([F])=[(a_1, a_2, \ldots)]-\left[\almostzero\right]\in  KO_0(A(Y)).$$
One can similarly define the index map $\sigma\colon  KO_n^{lf}(Y)\to KO_n(A(Y))$ when
$n\not\equiv  0 \modd 8$ with the help of suspensions
 (refer back to the section after
Theorem \ref{amenable}).

\bigskip
\noindent
The proof of the following
vanishing theorem contains some of the same elements as are found in
 Section 2, but now in the context of a noncompact manifold $M$.

\begin{thm}
\label{index-theorem-two}
Let $Y$ be a noncompact space with an admissible exhaustion $\{Y_i\}$. Let
$M$ be a noncompact manifold. Assume that there is a uniformly continuous proper coarse
 map $f\colon M\to Y$ with
an admissible exhaustion $\{M_i; M_{i,j}\}$ of $M$ such that each $M_i$ is a compact manifold
with boundary $\partial M_i$, $f^{-1}(Y_i)= M_i$, $f^{-1}(Y_{i,j})=M_{i,j}$ and
$f^{-1}(\partial Y_i)=\partial M_i$.
Suppose that $M$ is spin and let $D_M$ be the Dirac operator on $M$.
If $M$ admits a metric of uniform positive scalar curvature, then the index
$\sigma(f_\ast[D_M])$ of $D_M$
 is zero in $KO_\ast(A(Y))$, where $f_\ast\colon KO_\ast^{lf}(M)\to KO_\ast^{lf}(Y)$ is
the homomorphism induced by $f$.
\end{thm}

\begin{proof}
We assume that the dimension of $M$ is $0\modd 8$. The other cases can be
handled in a similar way with the help of suspensions
 (refer back to the section after
Theorem \ref{amenable}).

\medskip\noindent
Let $Y_i$, $M_i$, $Y_{i,j}$ and $M_{i,j}$ be given as in the statement of
the theorem. In this proof, all $C^\ast$-algebras are the maximal ones.

\bigno
Let $f$ be an odd smooth chopping function  on the real line satisfying the following conditions:
(1) $f(x)\rightarrow \pm 1$ as $x \rightarrow \pm \infty$; (2) $g=f^2-1 \in S(\mathbb{R})$, the space of Schwartz functions,
 (3) if $\widehat{f}$ and $\widehat{g}$ are the Fourier transforms
of $f$ and $g$, respectively, then
$\supp (\widehat{f})\subseteq [-1, 1]$ and $\supp (\widehat{g} ) \subseteq [-2,2]$.
As stated earlier, such an odd chopping
function exists (cf.~Roe \cite[Lemma 7.5]{Roe-1988}).

\bigno Let $D_M$ be the Dirac operator on $M$. 
We define $$F=
f( D_M )= \frac{1}{2\pi} \int_{-\infty}^{\infty}
 \widehat{f}(t) \exp(itD_M) dt .$$
Let $$\ind_L ([F]) = [P_{F}]- \left[\almostzero\right]
\in KO_0 (C_L^\ast(  M  )),$$ where $P_{F}$ is the idempotent
 in the matrix algebra of  $(C_L^\ast(M))^+$ as given in the definition of the local index.

\bigno
Recall that there
exists $\varep_0>0$  such that, for any closed subspace $Z$
of $Y$,  any operator on a $Z$-module with propagation less than or equal
to $\varep_0$ can be lifted to the universal cover of $Z$.
Define $P_F^{(j)}= \chi_{M_j} P_F \chi_{M_j},$ where $ \chi_{M_j}$ is the characteristic function of $M_j$.
 Let $n_0$ be the smallest  natural number such that $n_0> \frac{2}{\varep_0}.$
We write 
\begin{equation*}
\exp(itD_M)=\underbrace{
\exp\left(\frac{it}{n_0}D_M\right) \cdots \exp\left(\frac{it}{n_0}D_M\right)}_{n_0}.
\tag{$\ast$}
\end{equation*}
Let $j'>j$ be the smallest integer such that $$d(M-M_{j'}, M_j)> 10n_0 \varep_0.$$
Here $n_0\varep_0$ is roughly 2. 
Let $\widetilde{M}_{j'}$ be the universal cover of $M_{j'}$.
Using the formula for $P_F$ in terms of $\exp(itD_M)$, the identity $(\ast)$ and   
the fact that $ \exp(\frac{it}{n_0}D_M) $  has propagation less than $\varep_0$ for all $t\in [-2,2]$, we can lift  $ P_F^{(j)} $ to  an element
 $\widetilde{P}_F^{(j)}$ in $(C_L^\ast(\widetilde{M}_{j'})^{\pi_1(M_{j'})})^+$.

\bigno
For any $i<j$, let $m_{i,j}=i+[\frac{j-i}{2}]$  and $m'_{i,j}=i+[\frac{j-i}{4}],$ where $[\frac{j-i}{2}] $ and $[\frac{j-i}{4}]$
are respectively the integer parts of
$\frac{j-i}{2}$ and $\frac{j-i}{4}.$

\bigno
We define
$$P_F^{(i,j)}= \chi_{M_{ m'_{i,j},j}} P_F \chi_{M_{m'_{i,j},j}},$$
 where $ \chi_{M_{m'_{i,j},j}}$ is the characteristic function of $M_{m'_{i,j},j}$.
Let $v$ be an invertible element in the matrix algebra of $C_0(\reals^7)^+$ representing
the generator in $$KO_{-1}(C_0(\reals^7))\cong KO_0(C_0(\reals^8)).$$
Define
\begin{equation*}
\tau_{i,j}=v\otimes P^{(i,j)}_F(0)+I\otimes (I-P^{(i,j)}_F(0))
\end{equation*}
and 
\begin{equation*}
\tau_{i,j}'=v^{-1}\otimes P^{(i,j)}_F(0)+I\otimes (I-P^{(i,j)}_F(0)).
\end{equation*}
Define 
\begin{equation*}
x_{i,j}=\left(
\begin{array}{cc}
I& \tau_{i,j}\\0&I
\end{array}
\right)
\left(
\begin{array}{cc}
I& 0\\-\tau_{i,j}'&I
\end{array}
\right)
\left(
\begin{array}{cc}
I& \tau_{i,j}\\0&I
\end{array}
\right)
\left(
\begin{array}{rr}
0& -I\\I&0
\end{array}
\right)
\end{equation*}
and 
\begin{equation*}
u_{i,j}=x_{i,j}\almostzero x_{i,j}^{-1}.
\end{equation*}

\noindent
Let  $|j-i|$ be large enough such that $$d(M_i, M-M_{m_{i,j}})> 10 n_0 \varep_0 .$$
We define
 $v_{i,j} \in (S^7C^\ast (M_{i, m_{i,j}}))^+$ and
$ w_{i,j}\in (S^7C^\ast (M_{m_{i,j}, j}))^+$ by:
$$ v_{i,j}= \chi_{M_{i, m_{i,j}}} u_{i,j}\chi_{M_{i, m_{i,j}}},$$
$$w_{i,j}= \chi_{M_{ m_{i,j},j}} u_{i,j} \chi_{M_{ m_{i,j},j}},$$
where the identities in $(S^7C^\ast (M_{i, m_{i,j}}))^+$ and $(S^7C^\ast (M_{m_{i,j}, j}))^+$ are respectively identified with the multiplication operators by $\chi_{M_{i, m_{i,j}}}$ and $\chi_{M_{ m_{i,j},j}}$.
By the propagation of $P_F$ and the formula for $u_{i,j}$, we have
$$u_{i,j}=v_{i,j}\oplus w_{i,j}.$$
This equality is proved exactly in the same manner as Claims \ref{long-claim1} and \ref{long-claim2}.
From the definitions of $v_{i,j}$ and $w_{i,j}$, we then have
$\mathrm{prop}(v_{i,j})< 100n_0 \varep_0$ and
$\mathrm{prop}(w_{i,j})< 100 n_0\varep_0.$

\bigno
Let $\widetilde{M}_{i,j'}$ be  the universal cover of $M_{i,j'}$ and let
$\pi_{i,j'}$ be the covering map from
$\widetilde{M}_{i,j'}$ of $M_{i,j'}$.
Again using the identity $(\ast)$ and the formula for $P_F$ in terms of $\exp(itD_M)$ and the small propagation of $ \exp(\frac{it}{n_0}D_M) $, we can lift $ P_F^{(i,j)}$ to an element $ \widetilde{P}_F^{(i,j)}$ in $(C^\ast(\widetilde{M}_{i,j'})^{\pi_1(M_{i,j'})})^+$,
where $j'$ is defined as in the construction of  the lifting of  $P_F^{(j)}$.
Let $\widetilde{u}_{i,j}$ be the lifting of $u_{i,j}$ to $\widetilde{M}_{i,j'}$.
Let $|j-i|$ be large enough such that $$d(M_i, M-M_{m_{i,j}})> 100 n_0 \varep_0 .$$
We define $\widetilde{v}_{i,j} \in (S^7C^\ast (\widehat{M}_{i, m_{i,j}})^{\pi_1(M_{i,j'})})^+$ and
$ \widetilde{w}_{i,j}\in (S^7C^\ast (\widehat{M}_{m_{i,j}, j'})^{\pi_1(M_{i,j'})})^+$ to be the liftings of
$v_{i,j}$ and $w_{i,j}$ respectively, 
where $\widehat{M}_{i, m_{i,j}}=\pi_{i,j'}^{-1} (M_{i, m_{i,j}})$ 
and $\widehat{M}_{m_{i,j}, j'}= \pi_{i,j'}^{-1} ( M_{m_{i,j}, j'}    )$.
We have
\begin{equation*}
\widetilde{u}_{i,j}=\widetilde{v}_{i,j}\oplus \widetilde{w}_{i,j}.
\end{equation*}

\bigno
Next we shall represent the index class
$\sigma([D_M])$ as a $KO$-theory element explicitly constructed  using the above liftings.

\bigno
Let $\{j_k\}$ be a sequence of integers
such that $j_k>k$ for each $k$ and $j_k-k\rightarrow \infty$ as $k\rightarrow\infty$.
Let $z_k$ be   the image of $\widetilde{w}_{k,j_k}$ under the inclusion map from $(S^7C^\ast (\widehat{M}_{m_{k,j_k}, j_k'})^{\pi_1 (M_{k,j_k'})} )^+$ to
$(S^7C^\ast (\widetilde{M}_{k, j_k'})^{\pi_1 (M_{k,j_k'})} )^+$. Let $\pi_{j_k'}$ be the covering map from the universal cover of $\widetilde{M}_{j_k'} $ to $M_{j_k'}$.
Let $y_k$ be the element in the image of  the inclusion map from 
$(S^7C^\ast(\pi_{j_k'}^{-1} (M_{k, j_k'}))^{\pi_1 (M_{j_k'})})^+$ to 
$(S^7C^\ast(\widetilde{M}_{j_k'}))^{\pi_1 (M_{j_k'})})^+$ defined by
$$y_k=\phi_{k, j_k'} (z_k),$$ where $\phi_{k, j_k'}$ is the homomorphism
from  $(S^7C^\ast(\widetilde{M}_{k, j_k'} )^{\pi_1 (M_{k, j_k'})})^+$ to
 $(S^7C^\ast(\widetilde{M}_{j_k'} )^{\pi_1 (M_{j_k'})})^+$.
Note that the existence of this homomorphism follows from Lemma \ref{big-lemma} and Formula \ref{psi-max-formula}.

\bigno
As before, let $\psi_k$ be the natural map $\colon 
S^7C^\ast(\widetilde{M}_{j_k'})^{\pi_1(M_{j_k'})}
\to S^7C^\ast(\widetilde{M}_{j_{k+1}'})^{\pi_1(M_{j_{k+1}'})}$.
We similarly define $z'_k(\lambda)$  by replacing $P_{F}^{(j_k)}(0)$ with
 $(1-\lambda) P_{F}^{(j_k)}(0)+\lambda P_{F}^{(j_{k+1})}(0)$ in the definition of  $z_k$.
Define $y_k'(\lambda)=\phi_{k, j_k}(z_k'(\lambda))$.

\bigno
Let 
\begin{equation*}
\tau_k(\lambda)=v\otimes \left((1-\lambda)\psi_k(\widetilde{P}^{(j_k)}_F(0))
+\lambda \widetilde{P}^{(j_{k+1})}_F(0)\right) +I\otimes \left(I-
 \left((1-\lambda)\psi_k(\widetilde{P}^{(j_k)}_F(0))+\lambda
\widetilde{P}^{(j_{k+1})}_F(0)\right)\right)
\end{equation*}
and 
\begin{equation*}
\tau'_k(\lambda)=v^{-1}\otimes \left((1-\lambda)\psi_k(\widetilde{P}^{(j_k)}_F(0))
+\lambda \widetilde{P}^{(j_{k+1})}_F(0)\right) +I\otimes \left(I-
 \left((1-\lambda)\psi_k(\widetilde{P}^{(j_k)}_F(0))+\lambda
\widetilde{P}^{(j_{k+1})}_F(0)\right)\right)
\end{equation*}
for all $\lambda \in [0,1]$.
For all $s, \lambda\in [0,1]$, define $(w_k(s))(\lambda)$ to be the product
\begin{equation*}
\left(
\begin{array}{cc}
I& (1-s)\tau_k(\lambda)\\0&I
\end{array}
\right)
\left(
\begin{array}{cc}
I& 0\\-(1-s)\tau'_k(\lambda)&I
\end{array}
\right)
\left(
\begin{array}{cc}
I& (1-s)\tau_k(\lambda)\\0&I
\end{array}
\right)
\left(
\begin{array}{rr}
0& -I\\I&0
\end{array}
\right).
\end{equation*}

\medskip\noindent
Define
\begin{equation*}
(c_k(s))(\lambda) =(w_k(s))(\lambda)\almostzero ((w_k(s))(\lambda))^{-1}.
\end{equation*}

\bigskip\noindent
Note that, for each  $\lambda\in[0,1]$, the pair $(z'_k(\lambda), (c_k(\cdot))(\lambda))$ lies in 
$(S^7D_k^\ast)^+$, where $ D_k^\ast$ is as in the definition of $A(Y)$.

\bigno
Let $a_k= (z'_k, c_k)$.  By a homotopy invariance argument,  we have
$$\sigma([D_M])=[(a_1, a_2, \ldots)] -\left[\almostzero\right]\in  KO_0(A(Y)).$$

\noindent
In the above construction, for each $\alpha \geq 1$, we can replace respectively the Dirac  operator $D_M$ by $\alpha D_M$, $n_0$ by $[\alpha n_0]+1$, and $j_k'$ by another natural number $j_{k,\alpha}'$ satisfying
$d(M-M_{j_{k,\alpha}'}) >10\alpha n_0\varep_0$ to obtain the index of $\alpha D_M$:
$$\sigma([\alpha D_M])=[(a_{1, \alpha}, a_{2, \alpha}, \ldots)]
-\left[\almostzero\right] \in  KO_0(A(Y)).$$
Notice that the $KO$-theory class $\sigma([\alpha D_M])\in  KO_0(A(Y))$ is independent of the choice of $\alpha$.

\bigno
For all $k$, we
write $a_{k,\alpha}=(z'_{k,\alpha}, c_{k,\alpha}).$  Let $\tau_{k, \alpha}$ by replacing with $D$
with $\alpha D$ in the definition of $\tau_k$.
By the assumption that $M$  has uniform positive scalar curvature 
and the local nature of the Lichnerowicz formula, we have
\begin{equation*}
\tau_{k,\alpha}\to v\otimes \almostzero +I \otimes\almostzerotwo
\end{equation*}
in the operator norm 
when $\alpha \rightarrow \infty$.
This result implies the vanishing of $\sigma ([D_M])$.
\end{proof}

\bigno
Using the above notation
$D_i^\ast$, we
note that, by Guentner-Yu \cite{Guentner-Yu-2012}, there is a Milnor exact sequence given by
\begin{equation*}
0\to \varprojlim\,^{\!1}KO_\ast(D_i^\ast)\to KO_\ast(A(Y))\to \varprojlim
 KO_\ast(D_i^\ast)\to0. \tag{$\ast$}
\end{equation*}
This sequence gives rise to a commutative diagram
\begin{equation*}
\xymatrix{
0\ar[r]& \varprojlim\,^{\!1} KO_\ast(Y_i, \partial Y_i)\ar[r]\ar[d] & KO^{lf}_\ast(Y)\ar[r]\ar[d] &
\varprojlim KO_\ast(Y_i, \partial Y_i)\ar[r]\ar[d] &0\\
0\ar[r] & \varprojlim\,^{\!1} KO_\ast(D_i^\ast)\ar[r] & KO_\ast(A(Y))\ar[r]_\phi &
\varprojlim KO_\ast(D_i^\ast)\ar[r] &0
}
\end{equation*}
where the map $\phi\colon 
KO_\ast(A(Y))\to
\varprojlim KO_\ast(D_i^\ast)$ is induced by the $\ast$-homomorphism $\pi_i\colon
A(Y)\to D_i^\ast$ from 
\begin{equation*}
A(Y)\equiv \left\{a \in C\left([0, 1], \prod_{i=1}^\infty D_i^\ast\right)\colon
\rho(a(0))=a(1) \right\}
\end{equation*}
to $D_i^\ast$ obtained from the $i$-th component of the evaluation at 0.
We will use this diagram in the next section.

\bigskip
\section{A manifold with exotic positive scalar curvature behavior}

\bigno
We will now construct a noncompact manifold $M$
endowed with a nested exhaustion of compact subsets $M_i$,
 such that the $M_i$ can be endowed with
positive scalar metrics which are in totality incompatible in the sense
that $M$ itself has no metric of uniformly positive scalar curvature.

\bigno
In the last section we introduced a Milnor exact sequence with a $\limone$ term.
We quickly review some properties of this functor. If $\{G_i\}$ is an inverse system
of abelian groups indexed by the positive integers together with a coherent
family of maps $f_{j,i}\colon G_j\to G_i$ for all $j\ge i$, then $\limone G_i$
is categorically defined to be the first derived functor of $\varprojlim$. Eilenberg-Moore
\cite{Eilenberg-Moore-1962} also provides a description in the following. If $\Psi\colon
\prod G_i\to \prod G_i$ is defined by $\Psi(g_i)=(g_i-f_{i+1, i}(g_i))$, then
$\limone G_i$ is defined by $\limone G_i\equiv \mathrm{coker}(\Psi)$. Gray
\cite{Gray-1966} proves that, if each $G_i$ is countable, then $\limone G_i$
is either zero or uncountable.

\bigno
An example of an inverse system with a nontrivial $\limone$ term is
\begin{equation*}
S= \left\{\integers\xleftarrow{\times 3}\integers\xleftarrow{\times 3}\cdots\right\}
\end{equation*}
in which case we have the uncountable group
$\varprojlim^1S=\widehat{\integers}_3/\integers$. Let $\sphere^1$ denote the standard circle.
Consider the composite mapping cylinder $B_S$ of the infinite composite
\begin{equation*}
\sphere^1\leftarrow \sphere^1\leftarrow \sphere^1\leftarrow\cdots
\end{equation*}
which is capped off at the left end (see picture below),
where each map takes $z\in \sphere^1$ to $z^3$. 

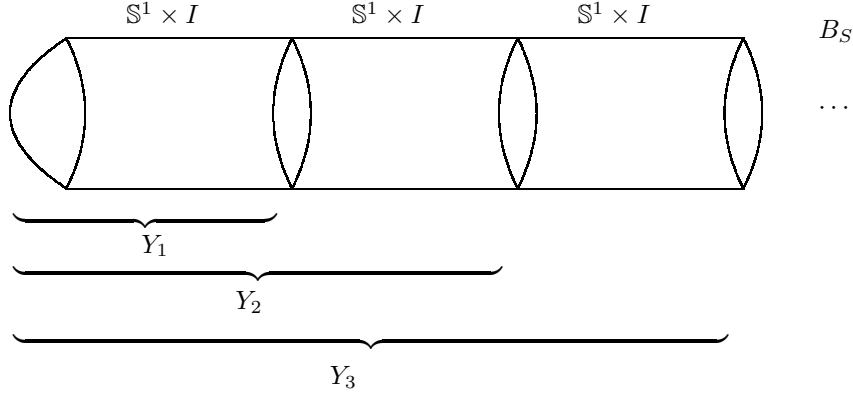
\begin{center}
\setlength{\unitlength}{1cm}
\begin{picture}(15,6)
\put(1,3.5){\line(1,0){9}}
\put(1,5.5){\line(1,0){9}}
\put(11, 4.5){$\cdots$}
\put(11, 5.5){$B_S$}
\qbezier(1,5.5)(-.5,4.5)(1,3.5)
\qbezier(1,5.5)(1.5,4.5)(1,3.5)
\qbezier(4,5.5)(4.5,4.5)(4,3.5)
\qbezier(4,5.5)(3.5,4.5)(4,3.5)
\qbezier(7,5.5)(7.5,4.5)(7,3.5)
\qbezier(7,5.5)(6.5,4.5)(7,3.5)
\qbezier(10,5.5)(10.5,4.5)(10,3.5)
\qbezier(10,5.5)(9.5,4.5)(10,3.5)
\put(1.8,5.7){$\sphere^1\times I$}
\put(4.8,5.7){$\sphere^1\times I$}
\put(7.8,5.7){$\sphere^1\times I$}
\put(.3,3.2){$\underbrace{\hskip 3.5cm}$}
\put(2, 2.65){$Y_1$}
\put(.3,2.5){$\underbrace{\hskip 6.5cm}$}
\put(3.25, 1.9){$Y_2$}
\put(.3,1.6){$\underbrace{\hskip 9.5cm}$}
\put(4.5, .9){$Y_3$}
\end{picture}
\end{center}

\bigno
Let $Y_j$ be the given exhaustion of $B_S$. For each $i$, let $\phi_i\colon 
(Y_{i+1}, \partial Y_{i+1})\to (Y_i, \partial Y_i)$ be the obvious collapse map.
Notice that $Y_i$ is contractible and
that $\partial Y_j$ is a circle for all $j$.
Consider the sequence
\begin{equation*}
0\to \varprojlim\,\!\! ^1 KO_n(Y_j, \partial Y_j)\to KO_n^{lf}(B_S)\to \varprojlim KO_n(Y_j,
\partial Y_j)\to 0.
\end{equation*}

\bigskip
\begin{prop}
The group $\varprojlim\,\!\! ^1 KO_0(Y_j, \partial Y_j)$ is nontrivial.
\end{prop}

\begin{proof}
We have an exact sequence
\begin{equation*}
KO_2(Y_i)\to KO_2(Y_i/\partial Y_i)\cong KO_2(Y_i, \partial Y_i)\to KO_1(\partial Y_i)
\to KO_1(Y_i)
\end{equation*}
By the contractibility of $Y_i$, we have $\widetilde{KO}_2(Y_i)=0$ and $KO_1(Y_i)=0$. Therefore
$KO_2(Y_i, \partial Y_i)\to KO_1(\partial Y_i)$ is an isomorphism. Consider
the commutative square:
\begin{equation*}
\xymatrix{
KO_2(Y_i, \partial Y_i) \ar[r]^\partial_\cong \ar[d]_{\phi_\ast} & KO_1(\partial Y_i)\ar[d]^{\times 3}\\
KO_2(Y_{i-1}, \partial Y_{i-1})\ar[r]^\partial_\cong  & KO_1(\partial Y_{i-1})\\
}
\end{equation*}
Clearly it follows that $\phi_\ast$ is multiplication by 3, so that at the level of quotients
we have a map $\phi_\ast\colon KO_2(Y_{i+1}, \partial Y_{i+1})\to KO_2(Y_i,
\partial Y_i)$.
\end{proof}

\begin{thm}
\label{conner}
Let $[c]\in KO^{lf}_\ast(B_S)$, where $B_S$ is endowed with an exhaustion by compact sets
$\{Y_i\}$ as above. There is $(M,f)\in \Omega^{spin}(B_S)$ such
that
\begin{enumerate}
\item $f_\ast[D_M]=[c]$;
\item the inverse images $(M_i, \partial M_i)=f^{-1}(Y_i, \partial Y_i)$
are compact manifolds with boundary such that the induced maps
  $\pi_1(M_i)\to \pi_1(Y_i)$ and $\pi_1(\partial M_i)\to \pi_1(\partial Y_i)$
are all isomorphisms.
\end{enumerate}
\end{thm}

\begin{proof}
%
%
We consider a proper map $g\colon W^n\to B_S$ with
 $n\ge5$ from a noncompact manifold $W^n$ to $B_S$.  Notice that every 
spin cobordism class in $\Omega^{spin}_n(pt)$ with $n\ge2$ has a simply connected 
manifold representative, 
and consequently for
any compact polyhedron $Z$ and any $n> \dim(Z)+1$, every class in $\Omega_n^{spin}(Z)$
has a manifold representative $N$ with $\pi_1(N) \cong \pi_1(Z)$. Now we apply this observation
 to all of the pieces of the exhaustion of $X$ (i.e. first apply it for the inverse images $g^{-1}(\sphere^1)$
of the separating circles, and then relatively to the inverse images 
$g^{-1}(Y_i-Y_{i-1})$ of the annular regions $Y_i-Y_{i-1}$). These pieces can be assembled to
procure the required $(M,f)\in \Omega^{spin}_n(B_S)$.
\end{proof}

\begin{thm}
\label{limone}
Let  $\xi$ be a nonzero class $\varprojlim\,\!\! ^1 KO_{n+1}(Y_j, \partial Y_j)$ and
consider $\xi$ also as an element of $KO_n^{lf}(B_S)$.
Let $M$ be as given in the above theorem with the exhaustion $(M_i, \partial M_i)$.
Then each $M_i$ has a metric of positive scalar curvature which
is collared at the boundary, but $M$ itself does not have a metric
of uniformly positive scalar curvature.
\end{thm}

\begin{proof}
We can choose the metric on $Y$ such that the map $f$ in Theorem \ref{conner} is a uniformly continuous proper coarse map.

\bigno
We have a commutative diagram
{\footnotesize
\begin{equation*}
\xymatrix{
0\ar[r] &
\varprojlim\,\!\! ^1 KO_{n+1}(M_i, \partial M_i)\ar[r]\ar[d] &
 KO_n^{lf}(M)\ar[r]\ar[d] &
\varprojlim KO_n(M_i, \partial M_i)\ar[r]\ar[d] &
0\\
0\ar[r] &
\varprojlim\,\!\! ^1 KO_{n+1}(Y_i, \partial Y_i)\ar[r]\ar[d]&
 KO_n^{lf}(B_S)\ar[r] \ar[d] &
\varprojlim KO_n(Y_i, \partial Y_i)\ar[r]\ar[d] &
0\\
0\ar[r] &
\varprojlim\,\!\! ^1 KO_{n+1}(D_i^\ast)\ar[r]&
 KO_n(A(B_S))\ar[r]
& \varprojlim KO_n(D_i^\ast)\ar[r] &
0.\\
}
\end{equation*}
}

\noindent
By the definition of $D_i^\ast$ and homotopy invariance, we have $$D_i^\ast \cong
C^\ast_{max}(\pi_1(Y_{i}), \pi_1(\partial Y_{i}))\otimes
\mathcal{K}.$$
Since the $Y_i$ are contractible and the $\partial Y_i$ are circles, 
 the outer vertical arrows from the second
to third row are isomorphisms by Theorem \ref{amenable},
so the map $KO_n^{lf}(B_S)\to KO_n(A(B_S))$
is also an isomorphism.
Note that by construction the element $\xi$ in $KO_n^{lf}(B_S)$ will lift
to the Dirac class $[D_M]$ in $KO_n^{lf}(M)$. By the commutativity
of the diagram, the image of $[D_M]$ is zero in $\varprojlim KO_n(M_i, \partial M_i)$
so it is zero in $\varprojlim KO_n(D_i^\ast)$. Therefore it is zero in
each $KO_n(D_i^\ast)$. 


\medskip\noindent
Finally we establish the veracity of the relative Gromov-Lawson-Rosenberg conjecture
in our case. Indeed, with respect to the multiplication map $f\colon\integers\xrightarrow{\times 3}
\integers$, we have $\Omega_n^{spin}(f)=\Omega_{n-1}(\ast)\otimes\integers_3$.
Consider the index map of the relative Dirac operator $D\colon \Omega^{spin}_n(f)
\to \integers_3$. When $n\not\equiv1\modd 4$, we have $\ker(D)=0$.
Otherwise,  when $n\equiv1\modd 4$, the kernel $\ker(D)$ is
$\sphere^1\times \mathrm{ker}(\Omega_{n-1}(\ast)\to\integers)$.
This latter kernel $\ker(\Omega_{n-1}(\ast)\to \integers)$ consists of 
closed manifolds of positive scalar curvature according to Stolz. 
Therefore there is a positive scalar curvature metric
on each piece $M_i$ of the exhaustion that is collared around the boundary. Moreover
the image of $[D_M]$ is nonzero in $KO_n(A(B_S))$ so $M$ itself has no
metric of uniformly positive scalar curvature by Theorem \ref{index-theorem-two}.
\end{proof}

\bigskip
\section{A manifold with uncountably many connected 
components of positive scalar curvature metrics}
\label{sec:components}

\bigno
In this section we use the previously developed theory to identify a connected
noncompact manifold $M$ such that $PS(M)$,
the space of complete positive scalar curvature metrics on $M$
equipped with the $C^\infty$-topology, has uncountably many connected components.
    
\bigno
In various spin cases, it can be shown
using index theory that $PS(M)$ has infinitely many concordance classes. In fact, one
can prove that the 7-sphere $\sphere^7$ is such a manifold  (see Gromov-Lawson
\cite{Gromov-Lawson-1980} or Lawson-Michelsohn \cite{Lawson-Michelsohn-1989}).
With separability
and the openness of positivity, it is an argument in point-set topology to see that $PS(M)$
has at most countably many components when $M$ is compact.
 These properties
may fail in the noncompact case. In the compact open topology, positivity is not necessarily
an open condition. In the uniform topology, we typically do not have separability.

\bigno
In the proof of the following
theorem, we refer the reader to the paper of Xie-Yu \cite[Theorem A]{Xie-Yu-2012}, which develops
the notion of a relative higher index $\ind D_{g_1, g_2}$ on a spin closed manifold $N$
with two Riemannian metrics $g_1$ and $g_2$. This relative 
index is defined to be the higher index of the Dirac operator  $D_{g_1, g_2}$  on the  infinite cylinder $N\times\reals$, where
the cross section $N\times \{x\}$ is endowed with $g_1$ if $x< -1$ and with
$g_2$ if $x> 1$ and the metric in $N\times [-1,1]$ can be chosen to be arbitrary.
The nonvanishing of this relative index in $KO_\ast(C^\ast_r\pi_1(N))$ gives information
about the concordance classes of positive scalar curvature metrics on $N$. In the case
when the manifold $M$ is not compact but has an admissible exhaustion by compact sets, a similar
theory shows that a relative higher index can be constructed in $KO_\ast(A(M))$, where
$A(M)$ is the algebra constructed in section \ref{sec:exhaustion}.

\bigno
Prior to the theorem we also make the following observation. 
 Let $\pi$ be a fixed finitely presented group with generators $g_1,
\ldots, g_s$ and relations $r_1, \ldots, r_t$. 
Let $n\ge5$
Execute $s$ successive 0-surgeries on
$\sphere^n$ to produce a manifold $K'$ with fundamental group $F_s$, the free group on
$s$ generators. The process of {\it surgery on maps} (see Ranicki \cite{Ranicki-2002} for
explicit details) shows that one can then perform a 1-surgery on $K'$ to produce a manifold
with fundamental group $F_s/\langle r_1\rangle$, where $\langle r_1\rangle$ is the
subgroup of $F_s$ normally generated by $r_1$. After performing these 1-surgeries
successively with respect to $r_2,\ldots, r_t$,
we obtain a manifold $K$ with fundamental group $\pi$. Since $K$
is constructed from the sphere $\sphere^n$ from
surgeries of codimesion at least 3, it follows from the Gromov-Lawson
surgery theorem \cite[Theorem A]{Gromov-Lawson-1980} that $K$ also has a metric of positive
scalar curvature.

%
%

\bigno
An easy example of a manifold with uncountably many components of positive scalar
curvature metrics is the disjoint union of countably many copies of $\sphere^7$. Here
we present a connected example.

\begin{thm}
\label{uncountable}
There is a connected noncompact manifold $M$ for which the set $PS(M)$ of components of positive
scalar curvature metrics on $M$ is uncountable.
\end{thm}

\begin{proof}
We provide a general construction that provides a host of examples. Let $W$ be an
$(n+1)$-dimensional  spin manifold with nontrivial higher $\widehat{A}$-genus. For example,
we may take $W$ to be the torus $T^{n+1}$. Let $\pi=\pi_1(W)$. By the discussion above, we
can produce a manifold $N^n$ with a positive
scalar curvature metric $\alpha$ and $\pi_1(N)=\pi$.
We can perform a 0-surgery on the disjoint union
of $N\times I$ and $W$ to create a connected manifold $X'$. Let $\pi'=\pi_1(X')$ with a classifying
map $\alpha\colon \pi'\to B\pi'$. Let $[\beta]$ represent the class of $\beta$ in  
$\Omega^{spin}_{n+1}(B\pi')$. 
Execute additional surgeries
(via surgery on maps) on
$X'$ to arrive at a manifold $X$ with fundamental group $\pi$
and two boundaries components both homeomorphic to $N$.
(In other words, since $\pi'=\pi\ast\pi$, we kill each element of $\pi'$ of the form
$g_1g_2^{-1}$, where $g_1$ and $g_2$ represent the same element of $\pi$.)
The fundamental group of $X$ may not be $\pi'$, but, since $X$ is cobordant to $X'$, there
is a map $X\to B\pi'$ which is also represented by $[\beta]$ in $\Omega^{spin}_{n+1}(B\pi')$.

\medskip\noindent
Let $\alpha'$ be the positive scalar curvature metric on $N$ as the other boundary component of $X$,
as constructed by the Gromov-Lawson surgery theorem.
Let $D_{\alpha, \alpha'}$ be the Dirac operator on $N\times\reals$. Here the Riemannian
metric on $N\times (-\infty, -1)$ is defined using the product metric
of $\alpha$ on $N$ and the standard metric on $(-\infty, -1)$, and the metric on
$N\times (1,\infty)$ is defined using the product metric of $\alpha$ on $N$
and the standard metric on $(1, \infty)$. The metric on $N\times [-1, 1]$ can be arbitrary.
From the nontriviality of the higher $\widehat{A}$-genus for $W$, we can infer
that the  relative higher index $\ind D_{\alpha, \alpha'}$ of $N$ is nonzero
in $KO_\ast(C_r^\ast\pi)$.

\medskip\noindent
This nonzero
condition implies that $\alpha$ and $\alpha'$ lie in different connected
components of $PS(N)$.
Form the infinite connected sum $M=N\# N\# \cdots$ with the obvious
exhaustion $M_i=(\underbrace{N\#\cdots \# N}_i)-\disk^n$. Apply the Gromov-Lawson
surgery theorem to modify the metric near each glueing so that $M$ is positively curved
at every point. 
On each summand $N$ we make a choice to endow
$N$ with either $\alpha$ or $\alpha'$.  Clearly the number of metrics
on $M$ constructed in this way is uncountable, and these metrics are
all in different connected components of $PS(M)$ by an application of our relative higher
index, which  lies in $KO_\ast(A(M))$ as explained in the following.
If $\beta, \beta'$ are
two distinct metrics on $M$ defined in this way, let $D_{\beta, \beta'}$ be
the Dirac operator on the product $M\times \reals$. Here the
metric on $M\times (-\infty, -1)$ is defined using the product metric
of $\alpha$ on $N$ and the standard metric on $(-\infty, -1)$, and the metric on
$M\times (1,\infty)$ is defined using the product metric of $\alpha$ on $N$
and the standard metric on $(1, \infty)$.
 The metric on $M\times [-1, 1]$ can be arbitrary. We can define
 a relative higher index  $\mathrm{ind}(D_{\beta, \beta'})$ 
 of $D_{\beta, \beta'}$ in $KO_\ast(A(M))$.  
By the relative higher index theorem in Xie-Yu \cite[Theorem A]{Xie-Yu-2012},  the relative higher index $\ind(D_{\alpha, \alpha'})$
does not lie in the image of the map $i_\ast\colon KO_\ast(\reals)
\to KO_\ast(C^\ast_r\pi_1(N))$, where $\reals$ is the one-dimensional
real  $C^\ast$-algebra and $i\colon \reals\to C_r^\ast\pi_1(N)$ is the 
inclusion map.  The Pimsner
Theorem (see \cite[Theorem 18]{Pimsner-1986}) allows us to compute $KO_\ast(D_i^\ast)$.
The above facts and the relative index theorem of \cite{Xie-Yu-2012} imply  that if $\beta, \beta'$ are
two distinct metrics on $M$ defined as above,
then
$(\pi_i)_\ast(\mathrm{ind}(D_{\beta, \beta'}))$ 
 is nonzero in $KO_\ast(D_i^\ast)$ for $i$ sufficiently large. Here 
 $\pi_i\colon A(M)\to D_i^\ast $ 
is the $\ast$-homomorphism defined in Section 3.
The Milnor sequence ($\ast$) given after Theorem \ref{index-theorem-two}
tells us that the index is nonzero in $KO_\ast(A(M))$. Therefore $\beta$ and
$\beta'$ are two metrics of $M$ that lie in different connected components
of $PS(M)$. Since $\beta$ and $\beta'$ are arbitrary, it follows that $PS(M)$
has uncountably many components.
\end{proof}

\bigno
In a sequel to this paper, we will provide examples of noncompact contractible spaces
with exotic positive scalar curvature behavior.

\bigskip\bigskip


\providecommand{\bysame}{\leavevmode\hbox to3em{\hrulefill}\thinspace}
\providecommand{\MR}{\relax\ifhmode\unskip\space\fi MR }
\providecommand{\MRhref}[2]{%
  \href{http://www.ams.org/mathscinet-getitem?mr=#1}{#2}
}
\providecommand{\href}[2]{#2}

\end{document}